\documentclass[12pt, a4paper]{article}
\usepackage[margin=2.4cm]{geometry}
\usepackage[english]{babel}

\usepackage[runin]{abstract}
\usepackage{titling}

\usepackage{amsmath}
\usepackage{amsthm}
\usepackage{amsfonts}
\usepackage{amssymb}
\usepackage{mathtools}
\usepackage{clrscode}
\usepackage{enumitem}
\usepackage{mdwlist}

\usepackage[nocompress]{cite}%---for extra cite features

\abslabeldelim{.}
\newcommand{\keywordsname}{Key words and phrases}
\makeatletter
\newcommand{\keywords}[1]{%
\def\thekeywords{#1}%
\begin{@bstr@ctlist}
\hspace*{\abstitleskip}{\abstractnamefont\keywordsname\@bslabeldelim}\abstracttextfont\
#1%
\par\end{@bstr@ctlist}
}
\makeatother
\newcommand{\subjclassname}{Mathematics subject classification}
\makeatletter
\newcommand{\subjclass}[2][2020]{%
\begin{@bstr@ctlist}
\hspace*{\abstitleskip}{\abstractnamefont\subjclassname\ (#1)\@bslabeldelim}\abstracttextfont\
#2%
\par\end{@bstr@ctlist}
}
\makeatother
\makeatletter
\def\and{%				%begin{tabular}
	\end{tabular}%
	and%
	\begin{tabular}[t]{c}}%
\makeatother
\makeatletter
\def\thanks#1{%\footnotemark
\protected@xdef\@thanks{\@thanks
\protect\footnotetext[\the\c@footnote]{#1}}%
}
\makeatother
\makeatletter
\let\addresses\@empty      %\let\thankses\@empty
\newcommand{\address}[2][]{\g@addto@macro\addresses{\address{#1}{#2}}}
\newcommand{\curraddr}[2][]{\g@addto@macro\addresses{\curraddr{#1}{#2}}}
\newcommand{\email}[2][]{\g@addto@macro\addresses{\email{#1}{#2}}}
\newcommand{\urladdr}[2][]{\g@addto@macro\addresses{\urladdr{#1}{#2}}}
\def\enddoc@text{%\ifx\@empty\@translators \else\@settranslators\fi
  \ifx\@empty\addresses \else\@setaddresses\fi}
\AtEndDocument{\enddoc@text}
\def\emailaddrname{E-mail address}
\def\@setaddresses{\par
  \nobreak \begingroup
  \interlinepenalty\@M
  \def\address##1##2{\begingroup%
    \par\addvspace\bigskipamount%\indent
    \@ifnotempty{##1}{(\ignorespaces##1\unskip) }%
    {\noindent\ignorespaces##2}\par\endgroup}%
  \def\email##1##2{\begingroup
    \@ifnotempty{##2}{\nobreak\noindent\emailaddrname
      \@ifnotempty{##1}{, \ignorespaces##1\unskip}\/:\space
      \ttfamily##2\par}\endgroup}%
  \addresses
  \endgroup
}

\makeatother
\makeatletter
\def\cstar#1{\expandafter\@cstar\csname c@#1\endcsname}
\def\@cstar#1{\ifcase#1\or $\ast$\or $\ast\ast$\or $\ast\ast\ast$\fi}
\AddEnumerateCounter{\cstar}{\@cstar}{$\ast\ast\ast$}
\makeatother

\newlist{conditions}{enumerate}{1}
\newlist{iconditions}{enumerate}{1}
\newlist{exconditions}{enumerate}{1}
\newlist{inthm}{enumerate}{1}
\setlist[conditions]{label=\normalfont(\alph*),ref=(\normalfont\alph*)}
\setlist[iconditions]{label=\normalfont(\roman*),ref=\normalfont(\roman*)}
\setlist[exconditions]{label=\normalfont(\roman*),ref=\normalfont(\roman*),wide,labelindent=0pt}
\setlist[inthm]{label=\normalfont(\thetheorem.\arabic*),ref=\normalfont(\thetheorem.\arabic*),wide,labelindent=0pt}

\mathchardef\mhyphen="2D
\newcommand{\CB}{\mathbb{C}}
\newcommand{\F}{\mathbb{F}}
\newcommand{\G}{\mathbb{G}}
\newcommand{\HB}{\mathbb{H}}
\newcommand{\PB}{\mathbb{P}}
\newcommand{\R}{\mathbb{R}}
\newcommand{\SB}{\mathbb{S}}
\newcommand{\V}{\mathbb{V}}
\newcommand{\Z}{\mathbb{Z}}

\newcommand{\C}{\mathcal{C}}
\newcommand{\RC}{\mathcal{R}}
\newcommand{\UC}{\mathcal{U}}

\newcommand{\Cinfty}{\C^{\infty}}
\newcommand{\pialg}{\pi^{\mathrm{alg}}}

\newcommand{\GL}{\mathrm{GL}}
\newcommand{\Mat}{\mathrm{Mat}}
\newcommand{\ON}{\mathrm{O}}
\newcommand{\SO}{\mathrm{SO}}
\newcommand{\Sp}{\mathrm{Sp}}
\newcommand{\SU}{\mathrm{SU}}
\newcommand{\U}{\mathrm{U}}

\newcommand{\Ker}{\operatorname{Ker}}

\newcommand{\tprod}[0]{\textstyle{\prod}} 

\newcommand{\real}[1]{_{\R}#1}
\newcommand{\Real}[2]{_{\R}#1_{#2}}

\newcommand{\map}[0]{\dasharrow}

\newtheorem{conjecture}{Conjecture}

\newtheorem{theorem}{Theorem}[section]
\newtheorem{corollary}[theorem]{Corollary}
\newtheorem{proposition}[theorem]{Proposition}
\newtheorem{lemma}[theorem]{Lemma}
\theoremstyle{definition}
\newtheorem*{acknowledgements}{Acknowledgements}
\newtheorem{definition}[theorem]{Definition}
\newtheorem{example}[theorem]{Example}
\newtheorem{notation}[theorem]{Notation}
\newtheorem{say}[theorem]{}

\theoremstyle{remark}
\newtheorem{case}{\indent Case}

\DeclarePairedDelimiter\abs{\lvert}{\rvert}%
\DeclarePairedDelimiter\norm{\lVert}{\rVert}%

\makeatletter
\let\oldabs\abs
\def\abs{\@ifstar{\oldabs}{\oldabs*}}
\let\oldnorm\norm
\def\norm{\@ifstar{\oldnorm}{\oldnorm*}}
\makeatother
\usepackage[pdftex]{hyperref}
\numberwithin{equation}{section}

\title{On approximation of maps into\\ real algebraic homogeneous spaces}
\date{}
\author{Jacek Bochnak \and Wojciech Kucharz}

\address{Jacek Bochnak\\Le Pont de l'\'Etang 8\\1323
Romainm\^otier\\Switzerland}
\email{jack3137@gmail.com}

\address{Wojciech Kucharz\\Institute of Mathematics\\Faculty of Mathematics and Computer
Science\\Jagiellonian University\\\L{}ojasiewicza 6\\30-348
Krak\'ow\\Poland}
\email{Wojciech.Kucharz@im.uj.edu.pl}

\begin{document}
\maketitle
\thispagestyle{empty}

\begin{abstract}
Let $X$ be a real
algebraic variety (resp. nonsingular real algebraic variety) and let $Y$ be
a homogeneous space for some linear real algebraic group. We prove that a continuous (resp.~$\Cinfty$) map $f \colon X \to
Y$ can be approximated by regular maps in the $\C^0$ (resp.~$\Cinfty$)
topology if and only if it is homotopic to a regular map. Taking $Y =
\SB^p$, the unit $p$-dimensional sphere, we obtain solutions of several problems 
that have been open since the 1980's and which concern approximation of 
maps with values in the unit spheres. This has several
consequences for approximation of maps between unit spheres. For
example, we prove that for every positive integer~$n$ every $\Cinfty$
map from $\SB^n$ into $\SB^n$ can be approximated by regular maps in the
$\Cinfty$ topology. Up to now such a result has only been known for five
special values of $n$, namely, $n=1,2,3,4$ or~$7$.
\end{abstract}

\keywords{Real algebraic variety, regular map, % 
approximation, homotopy, %
real algebraic group, homogeneous space, unit sphere.}
\hypersetup{pdfauthor={J. Bochnak, W. Kucharz},%
pdftitle={On approximation of maps into real algebraic homogeneous spaces},
pdfkeywords={\thekeywords}}
\subjclass{14P05, 14P25, 14P99.}

%Section 1
\section{Introduction}\label{sec:1}

In the present paper, we study approximation of continuous or $\Cinfty$
maps by real regular maps, that is, real algebraic morphisms. Our
results concern maps with values in homogeneous spaces for linear real
algebraic groups. Special attention is paid to maps into unit spheres.
The main result, Theorem~\ref{th-1-1}, and its consequences provide
answers to some approximation problems that have been open since the
1980's.

Throughout this work, by an \emph{algebraic variety} we always mean a
quasiprojective variety. To be precise, a \emph{real algebraic variety}
is a ringed space with structure sheaf of $\R$-algebras of $\R$-valued
functions, which is isomorphic to a Zariski locally closed subset of
real projective $n$-space $\PB^n(\R)$, for some $n$, endowed with the
Zariski topology and the sheaf of regular functions. This is  compatible
with \cite{bib2}, which contains a detailed exposition of real algebraic
geometry. We use the analogous definition of a \emph{complex algebraic
variety}, replacing $\R$ by~$\CB$. Recall that each real algebraic
variety in the sense used here is actually affine, that is, isomorphic
to an algebraic subset of $\R^n$, for some $n$, see
\cite[Proposition~3.2.10 and Theorem~3.4.4]{bib2}. Morphisms of
algebraic varieties are called \emph{regular maps} (in some of our
references they are called \emph{entire rational maps}, see \cite{bib3,
bib4, bib54, bib55}).

A \emph{real algebraic group} is a real algebraic variety $G$ endowed
with the structure of a group such that the group operations $G \times G
\to G$, $(a, b) \mapsto ab$, and $G \to G$, $a \mapsto a^{-1}$ are
regular maps. \emph{Morphisms of real algebraic groups} are regular maps
that are group homomorphisms. A real algebraic group is said to be
\emph{linear} if it is isomorphic to a Zariski closed subgroup of the
general linear group $\GL_n(\R)$, for some $n$. Obviously, the familiar
subgroups of $\GL_n(\R)$, the orthogonal group $\ON(n)$ and special
orthogonal group $\SO(n)$, are linear real algebraic groups. A
\emph{complex algebraic group}, linear or not, is defined analogously,
replacing $\R$ by $\CB$. Clearly, each real or complex algebraic group
$G$ is a nonsingular algebraic variety of pure dimension. Moreover, if
$G$ is linear, then each Zariski closed subgroup of $G$ is a linear
algebraic group.

Any complex algebraic variety $V$ carries the obvious underlying
structure $\real{V}$ of a real algebraic variety, called the
\emph{realification} of $V$ (for example, $\real{(\CB^n)} =
\R^{2n}$). If $G$ is a complex algebraic group, then its realification
$\real{G}$ is a real algebraic group with the same group operations as
in $G$. If $H$ is a Zariski closed subgroup of $G$, then $\real{H}$ is a
Zariski closed subgroup of $\real{G}$. Moreover, if the group $G$ is
linear, then so is the group $\real{G}$; indeed, it suffices to note that
the real algebraic group $\Real{\GL}{n}(\CB)$ is isomorphic to the image
of the real regular embedding
\begin{equation*}
\Real{\GL}{n}(\CB) \to \GL_{2n}(\R), \quad A + \sqrt{-1}\, B \mapsto 
\begin{pmatrix}
A & -B\\
B & A
\end{pmatrix}, 
\end{equation*}
where $A$, $B$ are real $n$-by-$n$ matrices. Consequently, each Zariski
closed subgroup of $\Real{\GL}{n}(\CB)$ is a linear real algebraic
group. In particular, the unitary group $\U(n)$ and special unitary
group $\SU(n)$ are linear real algebraic groups, being Zariski closed
subgroups of~$\Real{\GL}{n}(\CB)$.

The general linear group $\GL_n(\HB)$, where $\HB$ is the (skew)
field of quaternions, can also be viewed as a linear real algebraic
group via the standard embedding into $\GL_{4n}(\R)$. Therefore the
symplectic subgroup $\Sp(n)$ of $\GL_n(\HB)$ is a linear real algebraic
group.

Let $G$ be a real algebraic group. A \emph{$G$-space} (or a
\emph{$G$-variety}) is a real algebraic variety~$Y$ on which $G$ acts,
the action $G \times Y \to Y$, $(a,y) \mapsto a \cdot y$ being a
regular map. A \emph{homogeneous space} for $G$ is a $G$-space on which
$G$ acts transitively. Note that each homogeneous space for $G$ is a
nonsingular real algebraic variety of pure dimension.

Besides the Zariski topology, every real algebraic variety is endowed
with the Euclidean topology determined by the standard metric on $\R$.
Unless explicitly stated otherwise, all topological notions relating to
real algebraic varieties will refer to the Euclidean topology.

Given real algebraic varieties $X$ and $Y$, we denote by $\RC(X,Y)$ the
set of all regular maps from $X$ into $Y$. We regard $\RC(X,Y)$ as a
subset of the space $\C^0(X,Y)$ of all continuous maps endowed with the
$\C^0$ (that is, compact-open) topology. We say that a continuous map $f
\colon X \to Y$ can be \emph{approximated by regular maps in the $\C^0$
topology} if, for every neighborhood $\UC$ of $f$ in $\C^0(X,Y)$, there
is a regular map $g \colon X \to Y$ which belongs to $\UC$. If both
varieties $X$ and $Y$ are nonsingular, then $\RC(X,Y)$ is a subset of
the space $\Cinfty(X,Y)$ of all $\Cinfty$ maps endowed with the $\Cinfty$
topology (see \cite[p.~36]{bib32} or \cite[p.~311]{bib63} for the
definition of this topology and note that in \cite{bib32} it is called
the weak $\Cinfty$ topology); therefore the concept of approximation of
a $\Cinfty$ map $f \colon X \to Y$ by regular maps in the $\Cinfty$
topology is well-defined.

Our main result is the following.

\begin{theorem}\label{th-1-1}
Let $X$ be a real algebraic
variety (resp. nonsingular real algebraic
variety) and let $Y$ be a homogeneous space for some linear real algebraic
group. Then, for a continuous (resp. $\Cinfty$) map $f \colon X \to Y$,
the following conditions are equivalent:
\begin{conditions}
\item\label{th-1-1-a} $f$ can be approximated by regular maps in the
$\C^0$ (resp. $\Cinfty$) topology.

\item\label{th-1-1-b} $f$ is homotopic to a regular map.
\end{conditions}
\end{theorem}

The implication \ref{th-1-1-a}$\Rightarrow$\ref{th-1-1-b} holds because $X$ deformation retracts to some compact subset $K \subseteq X$ \cite[Corollary~9.3.7]{bib2}, and any two continuous maps from $K$ into $Y$ that are sufficiently close in the $\C^0$ topology are homotopic (the latter assertion is valid if $Y$ is an arbitrary $\Cinfty$ manifold). The proof of \ref{th-1-1-b}$\Rightarrow$\ref{th-1-1-a} is given in Section~\ref{sec:4} and depends on the
techniques developed in Sections~2 and~3. With notation as in Theorem~\ref{th-1-1}, if $X$ is a  compact nonsingular real algebraic curve, then, by the recent result of Benoist and Wittenberg \cite[Theorem~A(3)]{bib1a}, each $\Cinfty$ map from $X$ into $Y$ can be
approximated by regular maps in the $\Cinfty$ topology. Some special cases of
Theorem~\ref{th-1-1} have been anticipated since the 1980's, however,
the results obtained heretofore have been very incomplete.

As immediate consequences of Theorem~\ref{th-1-1} we get the following
two corollaries.

\begin{corollary}\label{cor-1-2}
Let $X$ be a real algebraic
variety (resp. nonsingular real algebraic
variety) and let $Y$ be a homogeneous space for some linear real
algebraic group. Then every continuous (resp. $\Cinfty$) null homotopic
map from $X$ to $Y$ can be approximated by regular maps in the $C^0$
(resp. $\Cinfty$) topology. \qed
\end{corollary}

\begin{corollary}\label{cor-1-3}
Let $Y$ be a homogeneous space for some linear real algebraic group. Then every continuous (resp. $\Cinfty$) map $Y \to Y$
that is homotopic to the identity map can be approximated by regular
maps in the $\C^0$ (resp. $\Cinfty)$ topology. \qed
\end{corollary}

It should be mentioned that no direct proofs of these corollaries, which
do not make use of Theorem~\ref{th-1-1}, are available.

In the rest of this section, we discuss the impact of
Theorem~\ref{th-1-1} within the framework described in the following
example.

\begin{example}\label{ex-1-4}
Here are some homogeneous spaces of interest.
\begin{exconditions}[widest=iii]
\item\label{ex-1-4-i} For every nonnegative integer $n$ the unit
$n$-sphere
\begin{equation*}
\SB^n = \{(x_0, \ldots, x_n) \in \R^{n+1} : x_0^2 + \cdots + x_n^2 = 1\}
\end{equation*}
is a homogeneous space for the real orthogonal group $\ON(n+1)$. 

\item\label{ex-1-4-ii} Let $\F$ stand for one of the fields $\R$, $\CB$
or $\HB$. Denote by $\G_r(\F^n)$ the Grassmannian of $r$-dimensional
$\F$-vector subspaces of $\F^n$, regarded as a real algebraic variety
(see \cite[pp.~72, 73, 352]{bib2}). Clearly, $\G_r(\F^n)$ is a
homogeneous space for the group $\GL_n(\F)$ viewed as a linear real
algebraic group.

\item\label{ex-1-4-iii} Denote by $\V_r(\F^n)$ the Stiefel manifold of
all orthonormal $r$-frames in $\F^n$, regarded as a real algebraic
variety. Clearly, $\V_r(\F^n)$ is a homogeneous space for the linear
real algebraic group $\ON(n)$ if $\F=\R$, $\U(n)$ if $\F=\CB$, and
$\Sp(n)$ if $\F=\HB$.

\item\label{ex-1-4-iv} Every real algebraic group $G$ is a homogeneous
space for $G$ under action by left translations.
\end{exconditions}
\end{example}

In view of Theorem~\ref{th-1-1} and Example~\ref{ex-1-4}\ref{ex-1-4-i},
we get at once the following result on maps into the unit $p$-sphere
$\SB^p$.

\begin{corollary}\label{cor-1-5}
Let $X$ be a real algebraic
variety (resp. nonsingular real algebraic
variety) and let $p$ be a positive integer. Then, for a continuous (resp.
$\Cinfty$) map $f \colon X \to \SB^p$, the following conditions are
equivalent:
\begin{conditions}
\item\label{cor-1-5-a} $f$ can be approximated by regular maps in the
$\C^0$ (resp. $\Cinfty$) topology.

\item\label{cor-1-5-b} $f$ is homotopic to a regular map.\qed
\end{conditions}
\end{corollary}

Until now Corollary~\ref{cor-1-5} with $X$ compact and $\dim X \geq p$ has only been
known for a few special values of $p$, namely, for $p=1,2$ or $4$ (see
\cite{bib3}) and $p=3$ or $7$ (see \cite{bib1}). Of course, if $\dim X <
p$, then the set of regular maps $\RC(X, \SB^p)$ is dense in the space
$\C^0(X, \SB^p)$ (resp. $\Cinfty(X, \SB^p)$). Indeed, $\SB^p$ with one
point removed is biregularly isomorphic to $\R^p$ via the stereographic
projection. Moreover, for $\dim X < p$, every continuous (resp.
$\Cinfty$) map from~ $X$ into~$\SB^p$ can be approximated in the space
$\C^0(X, \SB^p)$ (resp. $\Cinfty(X, \SB^p)$) by maps that are not
surjective, and hence the density assertion follows from the Weierstrass
approximation theorem.

The question which continuous or $\Cinfty$ maps are homotopic to regular
ones has been investigated in numerous works
\citeleft \citen{bib2}\citepunct \citen{bib4}\citepunct
\citen{bib7a}\citedash\citen{bib10}\citepunct
\citen{bib16}\citepunct \citen{bib18}\citepunct
\citen{bib28}\citedash\citen{bib30}\citepunct \citen{bib41}\citepunct
\citen{bib48}\citepunct \citen{bib50}\citepunct
\citen{bib54}\citedash\citen{bib56}\citepunct \citen{bib61}\citepunct
\citen{bib64}\citeright.
By applying Theorem~\ref{th-1-1} or Corollary~\ref{cor-1-5}, some of
these results can now be translated into the results on approximation by
regular maps. Next we give the first example of this principle.

\begin{theorem}\label{th-1-6}
For every positive integer $n$ the set of regular maps $\RC(\SB^n,
\SB^n)$ is dense in the space of $\Cinfty$ maps $\Cinfty(\SB^n,\SB^n)$.
\end{theorem}

\begin{proof}
Since each $\Cinfty$ map $\SB^n \to \SB^n$ is homotopic to a regular one
(see \cite[Theorem~1]{bib64} for $n$ odd, and \cite[Corollary~4.2]{bib4}
for $n$ arbitrary), the conclusion follows at once from
Corollary~\ref{cor-1-5}.
\end{proof}

Up to now Theorem~\ref{th-1-6} has only been known for $n=1,2$ or $4$
(see \cite{bib3}) and $n=3$ or~$7$ (see \cite{bib1}).

The following conjecture, although plausible, is wide open.

\begin{conjecture}\label{conj-i}
Let $Y$ be a homogeneous space for some linear real algebraic group.
Then, for every positive integer $n$, the set of regular maps
$\RC(\SB^n, Y)$ is dense in the space of $\Cinfty$ maps $\Cinfty(\SB^n,
Y)$. In particular, $\RC(\SB^n, \SB^p)$ is dense in $\Cinfty(\SB^n,
\SB^p)$ for every pair $(n,p)$ of positive integers.
\end{conjecture}

By \cite[Corollary~2.7]{bib3} (see also \cite[Theorems~2.2 and
11.1]{bib58}), Conjecture~\ref{conj-i} is valid for the homogeneous
space $Y = \G_r(\F^n)$ defined in Example~\ref{ex-1-4}\ref{ex-1-4-ii}.
Recall that $\G_1(\F^2)$ is biregularly isomorphic to the unit sphere
$\SB^{d(\F)}$, $d(\F)=\dim_{\R}\F$. In particular,
Conjecture~\ref{conj-i} holds if $Y$ is one of the spheres $\SB^1$,
$\SB^2$ or $\SB^4$. Presently Conjecture~\ref{conj-i} for maps between
unit spheres remains open, but it is supported further by
Theorem~\ref{th-5-6}.

In Section~\ref{sec:6} we obtain some results supporting
Conjecture~\ref{conj-i} for maps with values in the classical linear
real algebraic groups $\ON(m)$, $\SO(m)$, $\U(m)$, $\SU(m)$, $\Sp(m)$ or
in the Stiefel manifold $\V_r(\F^m)$ (see Propositions \ref{prop-6-1} and
\ref{prop-6-3}, and Theorem~\ref{th-6-2}).

In general, the equivalent conditions \ref{cor-1-5-a} and
\ref{cor-1-5-b} in Corollary~\ref{cor-1-5} are not satisfied. Indeed,
suppose given a compact nonsingular real algebraic variety $X$ such that
every regular map from $X$ into $\SB^p$ is null homotopic. Then, in view
of Corollary~\ref{cor-1-5}, a~$\Cinfty$ map $f \colon X \to \SB^p$ can
be approximated by regular maps if and only if it is null homotopic. Of
course, assuming that $\dim X = p$, there are always $\Cinfty$ maps from
$X$ into $\SB^p$ that are not null homotopic. By
\cite[Theorem~2.1]{bib5}, for ``most'' orientable nonsingular algebraic
hypersurfaces $X$ in $\PB^{2n+1}(\R)$, every regular map from~$X$ into~$\SB^{2n}$ is null homotopic. Moreover, according to \cite[Theorems 1.2
and~7.3]{bib10}, for ``most'' nonsingular real cubic curves~$C$
in~$\PB^2(\R)$, every regular map from $C^{2n}$ into $\SB^{2n}$ is null
homotopic. For real cubic curves one can also substitute other real
algebraic curves, see \cite[Theorem~1.3]{bib18} and
\cite[Corollary~2]{bib11}. We explicitly state the following.

\begin{theorem}\label{th-1-7}
Let $X$ be a compact connected nonsingular real algebraic variety of odd
dimension $k<2n$. Then a $\Cinfty$ map $f \colon X \times \SB^{2n-k} \to
\SB^{2n}$ can be approximated by regular maps in the $\Cinfty$ topology
if and only if it is null homotopic.
\end{theorem}

\begin{proof}
By \cite[Theorem~2.4]{bib5}, each regular map from $X \times \SB^{2n-k}$
into $\SB^{2n}$ is null homotopic, and hence the conclusion follows from
Corollary~\ref{cor-1-5}.
\end{proof}

The behavior of regular maps into odd-dimensional spheres is entirely
different.

\begin{theorem}\label{th-1-8}
Let $X$ be a compact connected oriented nonsingular real algebraic
variety of odd dimension $k$. Then either
\begin{iconditions}
\item\label{th-1-8-i} the set $\RC(X, \SB^k)$ is dense in the space
$\Cinfty(X, \SB^k)$, or

\item\label{th-1-8-ii} the closure of $\RC(X, \SB^k)$ in $\Cinfty(X,
\SB^k)$ coincides with the set
\begin{equation*}
\{ f \in \Cinfty(X, \SB^k) : \deg(f) \in 2\Z \},
\end{equation*}
where $\deg(f)$ is the topological degree of the map $f$.
\end{iconditions}
\end{theorem}

\begin{proof}
By \cite[Corollary~2.5]{bib4}, the set
\begin{equation*}
\{m \in \Z : m=\deg(g), \ g \in \RC(X, \SB^k)\}
\end{equation*}
is equal either to $\Z$ or $2\Z$, so the statement follows by combining
Hopf's theorem with Corollary~\ref{cor-1-5}.
\end{proof}

No example of an odd-dimensional variety $X$ for which
condition~\ref{th-1-8-ii} in Theorem~\ref{th-1-8} holds is known, which
raises the possibility that the following conjecture might be true.

\begin{conjecture}\label{conj-ii}
Let $X$ be a compact connected nonsingular real algebraic variety and
let~$k$ be a positive odd integer. If $\dim X = k$, then the set of
regular maps $\RC(X, \SB^k)$ is dense in the space of $\Cinfty$ maps
$\Cinfty(X, \SB^k)$.
\end{conjecture}

This conjecture is known to be true for $k=1$ (see
\cite[Corollary~1.5]{bib3}), but remains open for all other odd values
of $k$. Example~\ref{ex-5-8} shows that the assumption $\dim X = k$ in
Conjecture~\ref{conj-ii} cannot be replaced by $\dim X \geq k$.

According to the celebrated Nash--Tognoli theorem, every compact
$\Cinfty$ manifold~$M$ is diffeomorphic to a nonsingular real algebraic
variety~$X$, called an \emph{algebraic model} of~$M$, see \cite{bib53,
bib60} and \cite[Theorem~14.1.10]{bib2}.  Actually, $M$ admits an
uncountable family of pairwise nonisomorphic algebraic models, provided
$\dim M \geq 1$, see \cite{bib7a}. The behavior of regular maps
from $X$ into unit spheres, as $X$ runs through the class of algebraic
models of $M$, has been investigated in \citeleft
\citen{bib4}\citedash\citen{bib6}\citepunct \citen{bib10}\citepunct
\citen{bib41}\citeright. The following is a new result in this
direction.

\begin{theorem}\label{th-1-9}
Let $M$ be a compact connected $\Cinfty$ manifold of dimension~$m$. Then
there exists an algebraic model $X$ of $M$ such that the set of regular
maps $\RC(X, \SB^m)$ is dense in the space of $\Cinfty$ maps $\Cinfty(X,
\SB^m)$.
\end{theorem}

\begin{proof}
By \cite[Proposition~4.5]{bib4}, there exists an algebraic model $X$ of
$M$ such that each $\Cinfty$ map from $X$ into $\SB^m$ is homotopic to a
regular map. Hence the conclusion follows from Corollary~\ref{cor-1-5}.
\end{proof}

Recall that a compact oriented $\Cinfty$ manifold is said to be an
\emph{oriented boundary} if it is the boundary, with the induced
orientation, of a compact oriented $\Cinfty$ manifold with
boundary.

\begin{theorem}\label{th-1-10}
Let $M$ be a compact connected oriented $\Cinfty$ manifold of
dimension~$m$. Assume that either
\begin{iconditions}
\item\label{th-1-10-i} $m \equiv 2 \pmod 4$, or

\item\label{th-1-10-ii} $m \equiv 0 \pmod 4$ and the disjoint union of
two copies of $M$ is an oriented boundary.%\goodbreak
\end{iconditions}
Then there exists an algebraic model $X$ of $M$ such that each regular
map from $X$ into $\SB^m$ is null homotopic. Moreover, a $\Cinfty$ map
$f \colon X \to \SB^m$ can be approximated by regular maps in the
$\Cinfty$ topology if and only if it is null homotopic.
\end{theorem}

\begin{proof}
The first assertion is proved in \cite[Corollary~2.3]{bib41}, and the
second follows from Corollary~\ref{cor-1-5}.
\end{proof}

Let us comment on the boundary assumption in \ref{th-1-10-ii}. For each
integer $k \geq 4$ there exists a compact connected oriented $\Cinfty$
manifold of dimension $4k$, which is not an oriented boundary, but the
disjoint union of its two copies is; there is no such a manifold in
dimensions $4,8$ and $12$. This statement follows from the description
of the torsion elements in the oriented cobordism ring, see
\cite[p.~309]{bib62}.

We believe that in Theorem~\ref{th-1-10} (with $m \equiv 0 \pmod 4$) the
boundary condition is not only sufficient, but also necessary.

\begin{conjecture}\label{conj-iii}
Let $M$ be a compact connected oriented $\Cinfty$ manifold of dimension~$4k$. Assume that there exists an algebraic model $X$ of $M$ such that
each regular map from $X$ into $\SB^{4k}$ is null homotopic. Then the
disjoint union of two copies of~$M$ is an oriented boundary.
\end{conjecture}

This conjecture is known to be true if $\dim M = 4$, see
\cite[Theorem~3]{bib5}. Recall that if $\dim M = 4$, then $M$ is an
oriented boundary exactly when its signature $\sigma(M)$ is zero. If
$\sigma(M) \neq 0$ and $X$ is an algebraic model of~$M$, then each
$\Cinfty$ map $f \colon X \to \SB^4$ with $\deg(f)$ divisible by
$6\sigma(M)$ can be approximated by regular maps in the $\Cinfty$
topology, see \cite[Theorem~5.6]{bib5}.

Along with approximation by regular maps, investigated further in
\citeleft \citen{bib2}\citepunct \citen{bib3}\citepunct
\citen{bib6}\citepunct \citen{bib10}\citepunct
\citen{bib14}\citepunct\citen{bib15}\citepunct
\citen{bib18}\citepunct \citen{bib37}\citedash\citen{bib39a}\citepunct
\citen{bib48}\citepunct \citen{bib51}\citeright,
one can also consider approximation by stratified-regular (=~regulous =
continuous hereditarily rational = continuous rational, the last
equality holds assuming nonsingularity of varieties) maps \cite{bib40,
bib43, bib45, bib47, bib48, bib65} or by piecewise-regular maps
\cite{bib46}. There are remarkable similarities and differences in these
three cases.

We have already indicated how the paper is organized. It should be added
that the inspiration for Sections \ref{sec:2} and \ref{sec:3}, in which
we develop the main technical tools, comes from complex geometry, above
all Gromov's paper \cite{bib31} and the related works of Forstneri\v{c}
and others, see \cite{bib27} and the references therein.

The paper comes with the Appendix written by J\'anos Koll\'ar.

%Section 2
\section{Malleable real algebraic varieties}\label{sec:2}

The notions and results of this section will be generalized in the next
one. Nevertheless, it seems to be beneficial to discuss first the key
special case.

In what follows we work with vector bundles, which are always
$\R$-vector bundles. Let~$Y$ be a real algebraic variety. Given a vector
bundle $p \colon E \to Y$ over $Y$, with total space $E$ and bundle
projection $p$, we sometimes refer to $E$ as a vector bundle over~$Y$.
If $y$ is a point in $Y$, we let $E_y \coloneqq p^{-1}(y)$ denote the
fiber of $E$ over $y$ and write~$0_y$ for the zero vector in~$E_y$. As
usual, we call the set $Z(E) \coloneqq \{ 0_y \in E_y: y \in Y\}$ the
zero section of $E$.

The general theory of algebraic vector bundles over real algebraic
varieties is discussed in \cite[Section~12.1]{bib2}. For each algebraic
vector bundle $E$ over $Y$ there exists an algebraic vector bundle $E'$
over $Y$ such that the direct sum $E \oplus E'$ is algebraically
trivial, see \cite[Theorem~12.1.7]{bib2}. Moreover, if $E_1$ is an
algebraic vector subbundle of $E$, then there exists an algebraic vector
subbundle $E_2$ of $E$ such that $E_1 \oplus E_2 = E$. This is the case
since $E$~can be regarded as an algebraic vector subbundle of the
product vector bundle $Y \times \R^n$, for some $n$, and hence as $E_2$
one can take the orthogonal complement (with respect to the standard
scalar product on $\R^n$) of $E_1$ in $E$.

Assuming that $Y$ is a nonsingular real algebraic variety, we write $TY$
for the tangent bundle to $Y$ and $T_yY$ for the tangent space to $Y$ at $y
\in Y$.

\begin{definition}\label{def-2-1}
Let $Y$ be a nonsingular real algebraic variety.
\begin{iconditions}
\item\label{def-2-1-i} A \emph{spray} for $Y$ is a triple
$(E, p, s)$, where $p \colon E \to Y$ is an algebraic vector bundle
over~$Y$ and $s \colon E \to Y$ is a regular map 
such that $s(0_y) = y$ for all $y \in Y$.

\item\label{def-2-1-ii} A spray $(E, p, s)$ for $Y$ is said
to be \emph{dominating} if the derivative
\begin{equation*}
d_{0_y}s \colon T_{0_y}E \to T_yY
\end{equation*}
maps the subspace $E_y = T_{0_y}E_y$ of $T_{0_y}E$ onto $T_yY$, that is,
\begin{equation*}
d_{0_y}s(E_y) = T_y Y
\end{equation*}
for all $y \in Y$.

\item\label{def-2-1-iii} The variety $Y$ is called \emph{malleable} if
it admits a dominating spray.
\end{iconditions}
\end{definition}

The simplest malleable variety is $Y = \R^n$ which admits a dominating
spray $(E,p,s)$, where $p \colon E \coloneqq Y \times \R^n
\to Y$ is the product vector bundle and $s \colon E \to Y$
is defined by $s(y,v) = y + v$ for all $(y,v) \in E$. More substantial
and useful examples are provided by Examples \ref{ex-2-5} and
\ref{ex-2-6}, and above all by Propositions \ref{prop-2-8} and \ref{prop-2-8-new}.

In the following two lemmas we describe some basic properties of
dominating sprays.

\begin{lemma}\label{lem-2-2}
Any malleable nonsingular real algebraic variety $Y$ admits a dominating
 spray $(E,p,s)$ such that $p \colon E \to Y$ is an
algebraically trivial vector bundle.
\end{lemma}

\begin{proof}
Suppose that $(\tilde E, \tilde p, \tilde s)$ is a
dominating spray for $Y$. Choose an algebraic vector bundle
$\tilde p' \colon \tilde E' \to Y$ such that the direct sum $p \colon E
\coloneqq \tilde E \oplus \tilde E' \to Y$ is an algebraically trivial
vector bundle. Defining $s \colon E \to Y$ by $s(v,v') = \tilde s(v)$, we get a
dominating spray $(E,p,s)$ for $Y$.
\end{proof}

\begin{lemma}\label{lem-2-3}
Let $Y$ be a malleable nonsingular real algebraic variety and let
$(E,p,s)$ be a dominating spray for $Y$. Then there exists a
dominating spray $(\hat E, \hat p, \hat s)$ for $Y$
such that $\hat p \colon \hat E \to Y$ is an algebraic vector subbundle
of $p \colon E \to Y$, $\hat s = s|_{\hat
E}$, and the restriction
\begin{equation*}
d_{0_y} \hat s|_{\hat E_y} \colon \hat E_y \to T_y Y
\end{equation*}
of the derivative $d_{0_y} \hat s \colon T_{0_y} \hat E \to T_y Y$ is a
linear isomorphism for all $y \in Y$.
\end{lemma}

\begin{proof}
Let $\alpha \colon E \to TY$ be the morphism of algebraic vector bundles
over $Y$ defined by
\begin{equation*}
\alpha(v) = d_{0_{p(v)}} s(v) \quad \text{for all } v \in E.
\end{equation*}
Since $\alpha$ is a surjective morphism, its kernel $\Ker \alpha$ is an
algebraic vector subbundle of $E$. It follows that $E = \hat E \oplus
\Ker \alpha$ is the direct sum for some algebraic vector subbundle $\hat
p \colon \hat E \to Y$ of $p \colon E \to Y$. This completes the proof.
\end{proof}

As a consequence of Lemma~\ref{lem-2-3}, we obtain the following.

\begin{proposition}\label{prop-2-4}
Any malleable nonsingular real algebraic variety $Y$ admits a
dominating spray of the form $(TY, \pi,  \sigma)$, where
$\pi \colon TY \to Y$ is the tangent bundle to $Y$.
\end{proposition}

\begin{proof}
With notation as in Lemma~\ref{lem-2-3}, the vector bundle $\hat E$ is
isomorphic to the tangent bundle $TY$, which completes the proof.
\end{proof}

Next we identify some malleable real algebraic varieties.

\begin{example}\label{ex-2-5}
Let $G$ denote either the real orthogonal group $\ON(n)$ or the real
special orthogonal group $\SO(n)$. Then each homogeneous space $Y$ for
$G$ is malleable. This assertion can be established by direct
computation as follows.

Let $\Mat_n(\R) \cong \R^{n^2}$ be the $\R$-algebra of all $n$-by-$n$
matrices with real entries. Let $0$~and~$1$ stand for the zero matrix
and the identity matrix in $\Mat_n(\R)$. Using the determinant function
$\det$, we obtain a Zariski open neighborhood
\begin{equation*}
\Omega \coloneqq \{ x \in \Mat_n(\R) : \det(1+x) \neq 0 \}
\end{equation*}
of the zero matrix. We identify the tangent space to $G$ at $1$ with the
vector subspace $V$ of $\Mat_n(\R)$ comprising all skew symmetric
matrices; $V \cong \R^d$ where $d = n(n-1)/2$. Thus
\begin{equation*}
\varphi \colon V \to G \cap \Omega, \quad v \mapsto (1-v)(1+v)^{-1}
\end{equation*}
is a well-defined biregular isomorphism with $\varphi(0)=1$. 
We get a dominating
spray $(E,p,s)$ for $Y$, where $p \colon E \coloneqq Y \times V \to
Y$ is the product vector bundle and $s \colon E
\to Y$ is defined by $s(y,v) = \varphi(v) \cdot y$ for all $(y,v) \in
E$.
\end{example}

In particular, for unit spheres we get the following.

\begin{example}\label{ex-2-6}
For every positive integer $n$ the unit sphere $\SB^n$ is a malleable
real algebraic variety. Indeed, this is the case by Example~\ref{ex-2-5}
because $\SB^n$ is a homogeneous space for the real orthogonal group
$\ON(n+1)$. Actually, as observed by J\'anos Koll\'ar, one obtains directly a dominating spray $(T\SB^n, \pi, \sigma)$ for $\SB^n$, where $\pi \colon T\SB^n \to \SB^n$ is the tangent bundle to $\SB^n$ and $\sigma \colon T\SB^n \to \SB^n$ restricted to the tangent space $T_p\SB^n$ to $\SB^n$ at $p$, viewed as a hyperplane in~$\R^{n+1}$ passing through $p$, is the inverse of stereographic projection from $-p$ to $T_p\SB^n$.
\end{example}

For the proof of Theorem~\ref{th-1-1} it is essential to generalize
Example~\ref{ex-2-5}. First we recall some terminology and notation
which is compatible with \cite{bib19, bib24, bib57}. Let $V$ be an
algebraic \emph{$\R$-variety}, that is, a complex algebraic variety
defined over $\R$ (equivalently, $V$~can be viewed as a reduced
quasiprojective scheme over $\R$). The set of real points of~$V$ will be
denoted by $V(\R)$. We emphasize the distinction between algebraic
$\R$-varieties and real algebraic varieties. An \emph{algebraic
$\R$-group} is an algebraic $\R$-variety $\Gamma$ endowed with the
structure of a group such that the identity element of $\Gamma$ is in
$\Gamma(\R)$, and the group operations are regular maps defined over
$\R$. An algebraic $\R$-group is said to be \emph{linear} if, for some
positive integer~$n$, it is isomorphic over $\R$ to a Zariski closed
subgroup of $\GL_n(\CB)$ defined over~$\R$. If $\Gamma$ is an algebraic
$\R$-group, then $\Gamma(\R)$ is a real algebraic group. Moreover, if
$\Gamma$ is a linear algebraic $\R$-group, then $\Gamma(\R)$ is a linear
real algebraic group. Conversely, each linear real algebraic group~$G$
is up to isomorphism of the form $\Gamma(\R)$ for some linear algebraic
$\R$-group~$\Gamma$. Indeed, we may assume that $G$ is a Zariski closed
subgroup of $\GL_n(\R)$ and take as $\Gamma$ the Zariski closure of $G$
in $\GL_n(\CB)$, see \cite[Lemma~2.2.4]{bib57}. In view of this
discussion we can make use of Chevalley's paper \cite{bib24} in the
proof of Proposition~\ref{prop-2-8} below.

\begin{proposition}\label{prop-2-8}
Let $G$ be a linear real algebraic group. Then each homogeneous space~$Y$ for $G$ is a malleable variety.
\end{proposition}

\begin{proof}
Let $d = \dim G$ and let $G^0$ be the irreducible component of the
variety $G$ that contains the identity element $1$ of $G$. By a theorem
of Chevalley \cite[Theorem~2]{bib24} or \cite[Corollary~7.12]{bib19},
$G^0$ is a unirational variety, and hence there exist a Zariski open
subset~$U$ of~$\R^d$ and a regular map $\psi \colon U \to G^0$ such that
the image $\psi(U)$ is Zariski dense in $G^0$. Clearly, at some point $a
\in U$ the derivative $d_a \psi \colon \R^d \to T_{\psi(a)} G^0$ is a
linear isomorphism. Using a translation in $\R^d$ we may assume that
$a=0 \in U$. The map $\varphi \colon U \to G$, $v \mapsto
\psi(v)\psi(0)^{-1}$ is regular, $\varphi(0) = 1$, and the derivative of
$\varphi$ at $0$ is a linear isomorphism. As noted by J\'anos Koll\'ar in no.~\ref{Kollar-2} of the Appendix, taking the map
\begin{equation*}
    h \colon \R^d \to \R^d, \quad h(v) = \frac{cv}{1+\norm{v}^2},
\end{equation*}
where $c$ is a small positive real number, we obtain a regular map $g \coloneqq \varphi \circ h \colon \R^d \to G$ such that $g(0)=1$ and the derivative of $g$ at $0$ is a linear isomorphism. 
Now we obtain a~dominating spray $(E,p,s)$ for $Y$, where $p \colon E \coloneqq Y
\times \R^d \to Y$ is the product vector bundle 
and $s \colon E \to Y$ is defined by $s(y,v) = g(v) \cdot y$ for
all $(y,v) \in E$.
\end{proof}

In an earlier version of Proposition~\ref{prop-2-8} we only used the map $s(y,v)=\varphi(v)\cdot y$ defined on the Zariski open subset $Y \times U$ of $E$. This was sufficient for all our main results, but Koll\'ar's observation helped us to improve the presentation.

Prompted by some comments by Olivier Benoist, we next discuss a generalization of Proposition~\ref{prop-2-8}.

Let $Y$ be a real algebraic variety and let $G$ be a linear real algebraic group acting on $Y$. The $G$-space $Y$ is called \emph{good} if $Y$ is nonsingular and for each point $y \in Y$ the derivative of the map $G \to Y$, $a \mapsto a \cdot y$ at the identity element of $G$ is surjective. If $Y$ is homogeneous for $G$, then the $G$-space $Y$ is good. The converse is not true in general. Indeed, let $V$ be an algebraic $\R$-variety, with $V(\R)$ nonempty, and let $\Gamma$ be a linear algebraic $\R$-group acting on $V$. Then the real algebraic group $\Gamma(\R)$ acts on the real algebraic variety $V(\R)$. Clearly, if the action of $\Gamma$ on $V$ is transitive, then the $\Gamma(\R)$-space $V(\R)$ is good, but it can happen, as is well-known, that $V(\R)$ is not homogeneous for $\Gamma(\R)$.

\begin{proposition}\label{prop-2-8-new}
Let $G$ be a linear real algebraic group. Then each good $G$-space $Y$ is a malleable variety.
\end{proposition}

\begin{proof}
The same as that of Proposition~\ref{prop-2-8}.
\end{proof}

Proposition \ref{prop-2-8-new} is of interest in view of Theorems \ref{th-4-2} and \ref{th-4-3}.

%Section 3
\section{Malleable submersions}\label{sec:3}

This section is devoted to developing techniques used in the proof of
Theorem~\ref{th-1-1}. Of independent interest are Theorem~\ref{th-3-9}
and Corollary~\ref{cor-3-10}.

\begin{notation}\label{not-3-1}
Throughout the present section $X$, $Z$ are nonsingular real algebraic
varieties, and $h \colon Z \to X$ is a regular map which is surjective
and submersive. Furthermore, $V(h)$ denotes the algebraic vector
subbundle of the tangent bundle $TZ$ to $Z$ defined by
\begin{equation*}
V(h)_z = \Ker(d_z h \colon T_zZ \to T_{h(z)}X) \quad \text{for all } z
\in Z.
\end{equation*}
Observe that $V(h)_z$ is the tangent space to the fiber $h^{-1}(h(z))$.
\end{notation}

Let $U$ be an open subset of $X$. A map $f \colon U \to Z$ is called a
\emph{section} of $h \colon Z \to X$ if $h(f(x))=x$ for all $x \in U$. A
section which is a $\Cinfty$ map is called a \emph{$\Cinfty$ section}.
We say that a section $g \colon U \to Z$ is \emph{regular} if there
exist a Zariski open neighborhood $\tilde U \subseteq X$ of $U$ and a
section $\tilde g \colon \tilde U \to Z$ such that $\tilde g|_U = g$ and
$\tilde g$ is a regular map. Obviously, every regular section is a
$\Cinfty$ section. We say that a $\Cinfty$ section $f \colon U \to Z$
can be \emph{approximated by regular sections in the $\Cinfty$ topology}
if, for every neighborhood $\UC$ of $f$ in the space $\Cinfty(U,Z)$ of
$\Cinfty$ maps, there exists a regular section $g \colon U \to Z$ that
belongs to $\UC$. By a \emph{homotopy of $\Cinfty$ sections} we mean a
continuous map $F \colon U \times [0,1] \to Z$ such that $F_t \colon U
\to Z$, $x \mapsto F(x,t)$ is a $\Cinfty$ section for every $t \in
[0,1]$. Two $\Cinfty$ sections $f_0,f_1 \colon U \to Z$ are said to be
\emph{homotopic through $\Cinfty$ sections} if there exists a homotopy $F
\colon U \times [0,1] \to Z$ of $\Cinfty$ sections with $F_0=f_0$ and
$F_1=f_1$. To study approximation of $\Cinfty$ sections by regular
sections we need several notions and auxiliary results.\goodbreak

\begin{definition}\label{def-3-2}
Let $h \colon Z \to X$ be the submersion of Notation~\ref{not-3-1}.
\begin{iconditions}
\item\label{def-3-2-i}
A \emph{spray} for $h \colon Z \to X$ is a triple
$(E,p,s)$, where $p \colon E \to Z$ is an algebraic vector bundle
over $Z$ and $s \colon E \to Z$ is a regular map 
such that
\begin{equation*}
s(E_z) \subseteq h^{-1}(h(z)) \quad \text{and} \quad s(0_z)=z
\quad \text{for all } z \in Z.
\end{equation*}
\item\label{def-3-2-ii}
A spray $(E,p,s)$ for $h \colon Z \to X$ is said to be
\emph{dominating} if the derivative 
\begin{equation*}
    d_{0_z}s \colon T_{0_z}E \to T_zZ
\end{equation*}
maps the subspace $E_z = T_{0_z}E_z$ of $T_{0_z}E$ onto $V(h)_z$, that
is,
\begin{equation*}
d_{0_z}s(E_z) = V(h)_z
\end{equation*}
for all $z \in Z$.

\item\label{def-3-2-iii}
The submersion $h \colon Z \to X$ is called \emph{malleable} if it
admits a dominating spray.
\end{iconditions}
\end{definition}

Significant examples of malleable submersions are provided by combining
Proposition~\ref{prop-2-8} and Lemma~\ref{lem-4-1}. Note that if $X$ is
reduced to a point, then Definition~\ref{def-3-2} coincides with
Definition~\ref{def-2-1}. The next two lemmas are generalizations of
Lemmas~\ref{lem-2-2} and \ref{lem-2-3}.

\begin{lemma}\label{lem-3-3}
If the submersion $h\colon Z \to X$ is malleable, then it admits a
dominating spray $(E,p,s)$ such that $p \colon E \to Z$ is
an algebraically trivial vector bundle.
\end{lemma}

\begin{proof}
Suppose that $(\tilde E, \tilde p, \tilde s)$ is a
dominating spray for $h \colon Z \to X$. Choose an algebraic
vector bundle $\tilde p' \colon \tilde E' \to Z$ such that the direct
sum $p \colon E \coloneqq \tilde E \oplus \tilde E' \to Z$ is an
algebraically trivial vector bundle. Defining $s \colon E \to Z$ by $s(v,v') = \tilde s(v)$, we get a
dominating spray $(E,p,s)$ for $h \colon Z \to X$.
\end{proof}

\begin{lemma}\label{lem-3-4}
Suppose that $(E,p,s)$ is a dominating spray for the
submersion $h \colon Z \to X$. Then there exists a dominating
spray $(\hat E, \hat p, , \hat s)$ for $h \colon Z \to X$ such
that $\hat p \colon \hat E \to Z$ is an algebraic vector subbundle of $p
\colon E \to Z$, $\hat s = s|_{\hat E}$,
and the restriction
\begin{equation*}
d_{0_z} \hat s|_{\hat E_z} \colon \hat E_z \to V(h)_z
\end{equation*}
of the derivative $d_{0_z} \hat s \colon T_{0_z} \hat E \to T_z Z$ is a
linear isomorphism for all $z \in Z$.
\end{lemma}

\begin{proof}
Let $\alpha \colon E \to V(h)$ be a morphism of algebraic vector bundles
over $Z$ defined by
\begin{equation*}
\alpha(v) = d_{0_{p(v)}} s(v) \quad \text{for all } v \in E.
\end{equation*}
Since $\alpha$ is a surjective morphism, its kernel $\Ker \alpha$ is an
algebraic vector subbundle of~$E$. It follows that $E = \hat E \oplus
\Ker \alpha$ is the direct sum for some algebraic vector subbundle $\hat
p \colon \hat E \to Z$ of $p \colon E \to Z$. This completes the proof.
\end{proof}

\begin{notation}\label{not-3-5}
Suppose that $(E,p,s)$ is a dominating spray for the
submersion $h \colon Z \to X$. Let $U$ be an open subset of $X$ and let
$f \colon U \to Z$ be a $\Cinfty$ section of $h \colon Z \to X$. Denote
by $p_f \colon E_f \to U$ the pullback of the vector bundle $p \colon E
\to Z$ under the map $f \colon U \to Z$. Recall that $p_f \colon E_f \to
U$ is a $\Cinfty$ vector bundle over $U$, where
\begin{equation*}
E_f \coloneqq \{(x,v) \in U \times E : f(x) = p(v)\}, \quad p_f(x,v)=x.
\end{equation*}
We define
\begin{equation*}
s_f \colon E_f \to Z, \quad s_f(x,v)=s(v).
\end{equation*}
\end{notation}

\begin{lemma}\label{lem-3-6}
Using Notation~\ref{not-3-5}, assume that 
\begin{equation*}
d_{0_z}s|_{E_z} \colon E_z \to V(h)_z
\end{equation*}
is a linear isomorphism for all $z \in Z$. Then $s_f \colon E_f \to Z$
maps diffeomorphically an open neighborhood of the zero section $Z(E_f)$
in $E_f$ onto an open neighborhood of $f(U)$ in $Z$.
\end{lemma}

\begin{proof}
Let $x \in U$. The zero vector in the fiber $(E_f)_x$ is $(x,
0_{f(x)})$, where $0_{f(x)}$ is the zero vector in the fiber
$E_{f(x)}$. Since $s_f(x,0_{f(x)}) = s(0_{f(x)}) = f(x)$, it follows
that $s_f$ induces a diffeomorphism between $Z(E_f)$ and $f(U)$.
Moreover, the derivative
\begin{equation*}
d_{(x, 0_{f(x)})} s_f \colon T_{(x, 0_{f(x)})} E_f \to T_{f(x)}Z
\end{equation*}
is an isomorphism because
\begin{equation*}
d_{0_z} s|_{E_z} \colon E_z \to V(h)_z
\end{equation*}
is an isomorphism for all $z \in Z$. Consequently, $s_f$ is a local
diffeomorphism at the point $(x,0_{f(x)})$. Now the lemma follows from \cite[(12.7)]{bib23}.
\end{proof}

\begin{lemma}\label{lem-3-7}
Suppose that $(E, p, s)$ is a dominating spray for the
submersion $h \colon Z \to X$. Let $U$ be an open subset of $X$ and let
$F \colon U \times [0,1] \to Z$ be a homotopy of $\Cinfty$ sections of
$h \colon Z \to X$. Let $U_0$ be an open subset of $X$ whose closure
$\overline{U_0}$ is compact and contained in $U$. Let $t_0$ be a point
in $[0,1]$. Then there exist a neighborhood $I_0$ of $t_0$ in $[0,1]$
and a continuous map $\xi \colon U_0 \times I_0 \to E$ such that
\begin{inthm}[widest=3.7.2]
\item\label{lem-3-7-1} $p(\xi(x,t)) = F(x,t_0)$ for all $(x,t) \in U_0
\times I_0$,

\item\label{lem-3-7-2} $\xi(x,t_0) = 0_{F(x,t_0)}$ for all $x \in U_0$,

\item\label{lem-3-7-3} $s(\xi(x,t)) = F(x,t)$ for all $(x,t) \in U_0
\times I_0$,

\item\label{lem-3-7-4} for every $t \in I_0$ the map $U_0 \to E^0$, $x
\mapsto \xi(x,t)$ is of class $\Cinfty$.
\end{inthm}
\end{lemma}

\begin{proof}
By Lemma~\ref{lem-3-4}, we may assume without loss of generality that
\begin{equation*}
d_{0_z}s|_{E_z} \colon E_z \to V(h)_z
\end{equation*}
is an isomorphism for all $z \in Z$. Defining $f \colon U \to Z$ by
$f(x) = F(x, t_0)$, in view of Lemma~\ref{lem-3-6}, there exist an open
neighborhood $M \subseteq E_f$ of the zero section $Z(E_f)$ and an
open neighborhood $N \subseteq Z$ of $f(U)$ such that the restriction
$\sigma \colon M \to N$ of $s_f \colon E_f \to Z$ is a diffeomorphism.
Since $\overline{U_0}$ is a compact subset of $U$, we can choose an open
neighborhood~$I_0$ of $t_0$ in $[0,1]$ such that $F_t(U_0) \subseteq N$
for all $t \in I_0$. Therefore, for each $t \in I_0$, there exists a
unique $\Cinfty$ map $\zeta_t \colon U_0 \to E_f$ satisfying
$\zeta_t(U_0) \subseteq M$ and $F_t(x) = \sigma(\zeta_t(x))$ for all $x
\in U_0$. Writing $\zeta_t(x)$ as $\zeta_t(x) = (\alpha_t(x),
\xi_t(x))$, where $\alpha_t \colon U_0 \to U$ and $\xi_t \colon U_0 \to
E$ are $\Cinfty$ maps, we get
\begin{equation*}
f(\alpha_t(x)) = p(\xi_t(x)) \quad \text{and} \quad s(\xi_t(x)) =
F_t(x).
\end{equation*}
By Definition~\ref{def-3-2}\ref{def-3-2-i}, $s(\xi_t(x)) \in
h^{-1}(h(p(\xi_t(x))))$, and hence
\begin{equation*}
h(s(\xi_t(x))) = h(f(\alpha_t(x))) = \alpha_t(x).
\end{equation*}
On the other hand,
\begin{equation*}
h(s(\xi_t(x))) = h(F_t(x)) = x.
\end{equation*}
Consequently, $\alpha_t(x)=x$. It follows that $\zeta_t \colon U_0 \to
E_f$ is a $\Cinfty$ section, over $U_0$, of the vector bundle $p_f
\colon E_f \to U$. Clearly, $\zeta_{t_0}(U_0) \subseteq Z(E_f)$.
Furthermore, the map
\begin{equation*}
\zeta \colon U_0 \times I_0 \to E_f, \quad (x,t) \mapsto \zeta_t(x)
\end{equation*}
is continuous. Note that $\zeta(x,t) = (x,\xi(x,t))$, where
\begin{equation*}
\xi \colon U_0 \times I_0 \to E, \quad (x,t) \mapsto \xi_t(x)
\end{equation*}
is a continuous map with $p(\xi(x,t)) = f(x) = F(x,t_0)$ for all $(x,t)
\in U_0 \times I_0$. By construction, the map~$\xi$ satisfies conditions
\ref{lem-3-7-1}--\ref{lem-3-7-4}.
\end{proof}

\begin{lemma}\label{lem-3-8}
Assume that the submersion $h \colon Z \to X$ is malleable. Let $U$ be
an open subset of $X$ and let $F \colon U \times [0,1] \to Z$ be a
homotopy of $\Cinfty$ sections of $h \colon Z \to X$. Let $U_0$ be an
open subset of $X$ whose closure $\overline{U_0}$ is compact and
contained in $U$. Then there exist a dominating spray
$(E,p,s)$ for $h \colon Z \to X$ and a continuous map $\xi \colon
U_0 \times [0,1] \to E$ such that $p \colon E = Z \times \R^n \to Z$
is the product vector bundle and $\xi(x,t) = (F(x,0), \eta(x,t))$ for
all $(x,t) \in U_0 \times [0,1]$, where the map $\eta \colon U_0 \times
[0,1] \to \R^n$ satisfies
\begin{inthm}[widest=3.8.3]
\item\label{lem-3-8-1} $\eta(x,0)=0$ for all $x \in U_0$,

\item\label{lem-3-8-2} $s(F(x,0),\eta(x,t)) = F(x,t)$ for all $(x,t) \in
U_0 \times [0,1]$,

\item\label{lem-3-8-3} for every $t \in [0,1]$ the map $U_0 \to \R^n$,
$x \mapsto \eta(x,t)$ is of class $\Cinfty$.
\end{inthm}
\end{lemma}

\begin{proof}
By Lemma~\ref{lem-3-3}, the submersion $h \colon Z \to X$ admits a
dominating spray $(\tilde E, \tilde p, \tilde s)$
such that $\tilde p \colon \tilde E = Z \times \R^m \to Z$ is the product
vector bundle. In view of Lemma~\ref{lem-3-7} and the compactness of the
interval $[0,1]$ (see the Lebesgue lemma for compact metric spaces
\cite[p.~28, Lemma~9.11]{bib22}), there exists a partition $0=t_0 < t_1
< \cdots < t_k =1$ of $[0,1]$ such that for each $i=1,\ldots,k$ there
exists a continuous map $\xi^i \colon U_0 \times [t_{i-1},t_i] \to \tilde
E$ with the following properties:
\begin{itemize}
\item $\xi^i(x,t) = (F(x,t_{i-1}), \eta^i(x,t))$ for all $(x,t) \in U_0
\times [t_{i-1},t_i]$,

\item $\eta^i(x, t_{i-1}) = 0$ for all $x \in U_0$,

\item $\tilde s(F(x,t_{i-1}), \eta^i(x,t)) = F(x,t)$ for all $(x,t) \in
U_0 \times [t_{i-1}, t_i]$,

\item for every $t \in [t_{i-1}, t_i]$ the map $U_0 \to \R^m$, $x
\mapsto \eta^i(x,t)$ is of class $\Cinfty$.
\end{itemize}
For $i=1,\ldots,k$ we define recursively a dominating spray
$(E(i), p^{(i)}, s^{(i)})$ for $h \colon Z \to X$ by
\begin{equation*}
(E(i), p^{(i)}, s^{(i)}) = (\tilde E, \tilde p,\tilde s) \quad \text{if } i=1,
\end{equation*}
while for $i \geq 2$ we require
\begin{equation*}
p^{(i)} \colon E(i) = Z \times (\R^m)^i \to Z
\end{equation*}
to be the product vector bundle and set
\begin{equation*}
s^{(i)} \colon E(i) \to Z, \quad s^{(i)}(z, v_1,\ldots,v_i) =
s^{(1)}(s^{(i-1)}(z, v_1,\ldots,v_{i-1}), v_i),
\end{equation*}
where $z \in Z$ and $v_1,\ldots,v_i \in \R^m$.

In particular, $(E,p,s) \coloneqq (E(k), p^{(k)}, s^{(k)})$
is a dominating spray for $h \colon Z \to X$. Note that $p
\colon E = Z \times \R^n \to Z$ is the product vector bundle with $\R^n
= (\R^m)^k$. Now, consider a map
\begin{equation*}
\xi \colon U_0 \times [0,1] \to E, \quad \xi(x,t) = (F(x,0),
\eta(x,t)),
\end{equation*}
where $\eta \colon U_0 \times [0,1] \to \R^n = (\R^m)^k$ is defined by
\begin{equation*}
\eta(x,t) = (\eta^1(x,t), 0,\ldots,0)
\end{equation*}
for all $(x,t) \in U_0 \times [t_0,t_1]$, and
\begin{equation*}
\eta(x,t) = (\eta^1(x,t_1), \ldots, \eta^{i-1}(x,t_{i-1}), \eta^i(x,t),
0,\ldots, 0)
\end{equation*}
for all $(x,t) \in U_0 \times [t_{i-1}, t_i]$ with $i=2,\ldots,k$. One
readily checks that $\eta$ is a well-defined continuous map satisfying
\ref{lem-3-8-1}--\ref{lem-3-8-3}.
\end{proof}

The main result of this section is the following.

\begin{theorem}\label{th-3-9}
Assume that the submersion $h \colon Z \to X$ is malleable. Let $U$ be
an open subset of $X$ and let $f \colon U \to Z$ be a $\Cinfty$ section
of $h \colon Z \to X$ that is homotopic through $\Cinfty$ sections to a
regular section. Then 
$f$ can be approximated by regular sections of $h
\colon Z \to X$ in the $\Cinfty$ topology.
\end{theorem}

\begin{proof}
Let $F \colon U \times [0,1] \to Z$ be a homotopy of $\Cinfty$ sections
such that $F_0$ is a regular section and $F_1 = f$. Let $U_0$ be an open
subset of $X$ such that its closure $\overline{U_0}$ is compact
and contained in $U$,
and let $(E,p,s)$, $\xi \colon U_0 \times [0,1] \to E$, $\xi(x,t)
= (F(x,0), \eta(x,t))$, ${\eta \colon U_0 \times [0,1] \to \R^n}$ be as in
Lemma~\ref{lem-3-8}. In particular, we have
\begin{equation*}
s(F(x,0), \eta(x,1)) = F(x,1) = f(x) \quad \text{for all } x \in U_0.
\end{equation*}
By the Weierstrass approximation theorem, there exists a regular map $\beta \colon X \to \R^n$ such that the restriction $\beta|_{U_0}$ is arbitrarily close to the $\Cinfty$ map $\eta_1 \colon U_0 \to \R^n$, $x \mapsto \eta(x,1)$ in the $\Cinfty$ topology. Then
\begin{equation*}
    g \colon U \to Z, \quad x \mapsto s(F(x,0), \beta(x))
\end{equation*}
is a regular map such that $g|_{U_0}$ is close to $f|_{U_0}$ in the $\Cinfty$ topology. Moreover, in view of Definition~\ref{def-3-2}\ref{def-3-2-i}, $g \colon U \to Z$ is a section of $h \colon Z \to X$. The proof is complete since $U_0$ is chosen in an arbitrary way.
\end{proof}

It is worthwhile to point out the following special case of
Theorem~\ref{th-3-9}.

\begin{corollary}\label{cor-3-10}
Assume that the submersion $h \colon Z \to X$ is malleable. Let $f \colon X \to Z$ be a $\Cinfty$ section of
$h \colon Z \to X$ that is homotopic through $\Cinfty$ sections to a
regular section. Then $f$ can be approximated by regular sections of $h
\colon Z \to X$ in the $\Cinfty$ topology. \qed
\end{corollary}

%Section 4
\section{Proofs of Theorem~\ref{th-1-1} and related results}\label{sec:4}

To begin with we discuss approximation of maps with values in a
malleable real algebraic variety. In order to make use of
Theorem~\ref{th-3-9} or Corollary~\ref{cor-3-10} we need the following
observation.

\begin{lemma}\label{lem-4-1}
Let $X,Y$ be nonsingular real algebraic varieties. Assume that the
variety $Y$ is malleable. Then the canonical projection
\begin{equation*}
h \colon X \times Y \to X, \quad (x,y) \mapsto x
\end{equation*}
is a malleable submersion.
\end{lemma}

\begin{proof}
Let $(E,p,s)$ be a dominating spray for $Y$. We obtain a
dominating spray $(\tilde E, \tilde p, \tilde s)$
for $h \colon X \times Y \to X$ setting
\begin{align*}
&\tilde E = \{ ((x,y),v) \in (X \times Y) \times E : y=p(v)\},\\
&\tilde p \colon \tilde E \to X \times Y, \quad ((x,y),v) \mapsto
(x,y),\\
&\tilde s \colon \tilde E \to X \times Y, \quad ((x,y),v) \mapsto
(x,s(v)).\qedhere
\end{align*}
\end{proof}

\begin{theorem}\label{th-4-2}
Let $X$ be a nonsingular real algebraic variety and let $Y$ be a
malleable nonsingular real algebraic variety. Then, for a $\Cinfty$ map
$f \colon X \to Y$, the following conditions are equivalent:
\begin{conditions}
\item\label{th-4-2-a} $f$ can be approximated by regular maps in the
$\Cinfty$ topology.

\item\label{th-4-2-b} $f$ is homotopic to a regular map.
\end{conditions}
\end{theorem}

\begin{proof}
It suffices to prove \ref{th-4-2-b}$\Rightarrow$\ref{th-4-2-a}. To this
end let $\Phi \colon X \times [0,1] \to Y$ be a homotopy such that
$\Phi_0$ is a regular map and $\Phi_1=f$. We may assume that $\Phi$ is a
$\Cinfty$ map, see \cite[Proposition~10.22]{bib49a}. By
Lemma~\ref{lem-4-1}, the canonical projection $h \colon X \times Y \to
X$ is a malleable submersion. Clearly,
\begin{equation*}
F \colon X \times [0,1] \to X \times Y, \quad (x,t) \mapsto (x,
\Phi(x,t))
\end{equation*}
is a homotopy of $\Cinfty$ sections of $h \colon X \times Y \to X$.
Therefore, according to Corollary~\ref{cor-3-10}, the $\Cinfty$ section
\begin{equation*}
X \to X \times Y, \quad x \mapsto (x,f(x))
\end{equation*}
can be approximated by regular sections in the $\Cinfty$ topology, which
implies that \ref{th-4-2-a} holds.
\end{proof}

The following variant of Theorem~\ref{th-4-2} can be derived from
Lemma~\ref{lem-3-8}.

\begin{theorem}\label{th-4-3}
Let $X$ be a real algebraic variety (possibly singular) and let
$Y$ be a malleable nonsingular real algebraic variety. Then, for a
continuous map $f \colon X \to Y$, the following conditions are
equivalent:
\begin{conditions}
\item\label{th-4-3-a} $f$ can be approximated by regular maps in the
$\C^0$ topology.

\item\label{th-4-3-b} $f$ is homotopic to a regular map.
\end{conditions}
\end{theorem}

\begin{proof}
It suffices to prove \ref{th-4-3-b}$\Rightarrow$\ref{th-4-3-a}. Suppose
that \ref{th-4-3-b} holds, and let $\gamma \colon X \to Y$ be a regular map
homotopic to $f$. We may assume that $X,Y$ are Zariski closed subsets of
$\R^k,\R^l$, respectively. Then there exists a Zariski open neighborhood
$\Omega \subseteq \R^k$ of $X$ and a regular map $\tilde \gamma \colon \Omega
\to \R^l$ with $\tilde \gamma|_X = \gamma$, see \cite[p.~62]{bib2}. By the Tietze
extension theorem, there exists a continuous map $\tilde f \colon \Omega
\to \R^l$ with $\tilde f|_X = f$. Now, let $\rho \colon T \to Y$ be a
$\Cinfty$ tubular neighborhood of $Y$ in $\R^l$, where $T$ is an open
neighborhood of $Y$ in $\R^l$ and $\rho$ is a $\Cinfty$ retraction.
Choose an open neighborhood $U \subseteq \Omega$ of $X$ such that
$\tilde f(U) \subseteq T$ and $\tilde \gamma(U) \subseteq T$. By \cite[p.~44, Theorem~2.2]{bib32}, we can also
choose a $\Cinfty$ map $U \to T$ arbitrarily close to $\tilde f|_U
\colon U \to T$ in the strong $\C^0$ topology; 
such a map is homotopic to $\tilde f|_U \colon U \to T$. Therefore, for
the proof of \ref{th-4-3-a}, we may assume that the map $\tilde f|_U
\colon U \to T$ is of class $\Cinfty$. Clearly, the maps $\varphi, \psi
\colon U \to Y$, defined by
\begin{equation*}
\varphi(x) = \rho(\tilde \gamma(x)), \quad \psi(x) = \rho(\tilde f(x)) \quad
\text{for all } x \in U,
\end{equation*}
are of class $\Cinfty$ and satisfy $\varphi|_X=\gamma$, $\psi|_X=f$.
Shrinking $U$, the variety $X$ is a continuous retract of $U$
($X$ is locally contractible \cite[Theorem~9.3.6]{bib2}, so it is a Euclidean neighborhood retract by Borsuk's theorem \cite[p.~537, Theorem~E.3]{bib22}).
Consequently, the $\Cinfty$ maps $\varphi, \psi$ are homotopic, and hence
there exists a $\Cinfty$ homotopy $\Phi \colon U \times [0,1] \to Y$
with $\Phi_0 = \varphi$ and $\Phi_1 = \psi$ \cite[Proposition~10.22]{bib49a}. Let $U_0 \subseteq \R^k$ be
an open subset whose closure $\overline{U_0}$ is compact
and contained in $U$.

Consider the canonical projection $h \colon \R^k \times Y \to \R^k$ and
the homotopy of its $\Cinfty$ sections
\begin{equation*}
F \colon U \to \R^k \times Y, \quad x \mapsto (x,\Phi(x,t)).
\end{equation*}
According to Lemma~\ref{lem-4-1}, $h \colon \R^k \times Y \to \R^k$ is a
malleable submersion. Hence, by Lemma~\ref{lem-3-8}, there exist a
dominating spray $(E,p,s)$ for $h \colon \R^k \times Y \to
\R^k$ and a continuous map ${\xi \colon U_0 \times [0,1] \to E}$ such that
$p \colon E = (\R^k \times Y) \times \R^n \to \R^k \times Y$ is the
product vector bundle and $\xi(x,t) = (F(x,0), \eta(x,t))$, where the
continuous map $\eta \colon U_0 \times [0,1] \to \R^n$ satisfies
\begin{equation*}
s(F(x,0), \eta(x,t)) = F(x,t) \quad \text{for all } (x,t) \in U_0 \times
[0,1].
\end{equation*}
In particular,
\begin{equation*}
s(F(x,0), \eta(x,1)) = F(x,1) = (x, \psi(x)) \quad \text{for all } x \in
U_0.
\end{equation*}
By the Weierstrass approximation theorem, there exists a regular map $\beta \colon \R^k \to \R^n$ such that the restriction $\beta|_{U_0}$ is arbitrarily close to the $\Cinfty$ map $\eta_1 \colon U_0 \to \R^n$, $x \mapsto \eta(x,1)$ in the $\C^0$ topology. Then
\begin{equation*}
    \alpha \colon U \to \R^k \times Y, \quad x \mapsto s(F(x,0), \beta(x))
\end{equation*}
is a $\Cinfty$ section of $h \colon \R^k \times Y \to \R^k$ such that $\alpha|_{U_0}$ is close in the $\C^0$ topology to the section
\begin{equation*}
    U_0 \to \R^k \times Y, \quad x \mapsto (x, \psi(x)).
\end{equation*}
We have $\alpha(x) = (x, \theta(x))$, where $\theta \colon U \to Y$ is a $\Cinfty$ map. Since $F(x,0) = (x, \varphi(x))$ and $\varphi|_X = \gamma$ is a regular map, the restriction $g \coloneqq \theta|_X \colon X \to Y$ is a regular map. By construction, $g|_{X \cap U_0}$ is close to the continuous map $\psi|_{X \cap U_0} = f|_{X \cap U_0}$. The proof is complete since $U_0$ is chosen in an arbitrary way.
\end{proof}

Now our main theorem follows immediately.

\begin{proof}[Proof of Theorem~\ref{th-1-1}]
By Proposition~\ref{prop-2-8}, the variety $Y$ is malleable. Hence the
conclusion follows from Theorems \ref{th-4-2} and \ref{th-4-3}.
\end{proof}

Theorem~\ref{th-1-1} holds, in particular, for maps with values in an
arbitrary linear real algebraic group~$G$. The assumption that $G$ is
linear cannot be omitted. Indeed, each linear real algebraic group is up
to isomorphism of the form $\Gamma(\R)$ for some linear algebraic
$\R$-group $\Gamma$ (see Section~\ref{sec:2}). In the category of
algebraic $\R$-groups the linear ones can be characterized as follows.

\begin{theorem}\label{th-4-4}
Let $k$ be a positive integer and let $\Gamma$ be an irreducible
algebraic $\R$-group. Then the following conditions are equivalent:
\begin{conditions}
\item\label{th-4-4-a} $\Gamma$ is a linear algebraic $\R$-group.

\item\label{th-4-4-b} Every continuous null homotopic map from $\SB^k$
into $\Gamma(\R)$ can be approximated by regular maps in the $\C^0$
topology.
\end{conditions}
\end{theorem}

\begin{proof}
By Theorem~\ref{th-1-1}, \ref{th-4-4-a} implies \ref{th-4-4-b}. To prove
the reversed implication suppose that \ref{th-4-4-a} does not hold.
Then, by Chevalley's theorem \cite{bib25} (see \cite{bib26} for a modern
treatment), there exists a surjective morphism of algebraic $\R$-groups
$\varphi \colon \Gamma \to A$, where $A$ is an Abelian $\R$-variety,
$\dim A \geq 1$. Since $\varphi$ is defined over $\R$, its restriction
$\varphi(\R) \colon \Gamma(\R) \to A(\R)$ is a regular map of real
algebraic varieties. The image $\varphi(\R)(\Gamma(\R))$ is Zariski
dense in $A(\R)$, the map $\varphi$ being surjective. Hence there exists
a point $a \in \Gamma(\R)$ at which the derivative $d_a \varphi(\R)
\colon T_a\Gamma(\R) \to T_{\varphi(\R)(a)}A(\R)$ is surjective.
Therefore, in view of the rank theorem for $\Cinfty$ maps, we can find a
continuous null homotopic map $f \colon \SB^k \to \Gamma(\R)$ such that
the composite map $\varphi(\R) \circ f \colon \SB^k \to A(\R)$ is not
constant. It follows that $\varphi(\R) \circ f$ cannot be approximated
by regular maps in the $\C^0$ topology because each regular map from
$\SB^k$ into $A(\R)$ is constant, see \cite[Corollary~3.9]{bib52}.
Consequently, \ref{th-4-4-b} does not hold. In other words,
\ref{th-4-4-b} implies \ref{th-4-4-a}.
\end{proof}

%Section 5
\section{Further results on regular maps into unit spheres}\label{sec:5}

The first result of this section can be viewed as a generalization of
Theorem~\ref{th-1-6}.

\begin{theorem}\label{th-5-1}
Let $X$ be a compact connected oriented nonsingular real algebraic
variety of dimension $n$. Then the set of regular maps $\RC(X, \SB^n)$
is dense in the space of $\Cinfty$ maps $\Cinfty(X, \SB^n)$ if and only
if there exists a regular map from $X$ into $\SB^n$ of topological
degree~$1$.
\end{theorem}

\begin{proof}
Suppose that $f \colon X \to \SB^n$ is a regular map with $\deg(f)=1$.
By \cite[Corollary~4.2]{bib4}, for every integer $d$ there exists a
regular map $\varphi_d \colon \SB^n \to \SB^n$ with $\deg(\varphi_d)=d$.
Since the composite map $\varphi_d \circ f$ is regular and
$\deg(\varphi_d \circ f) = d$, it follows from Hopf's theorem that each
$\Cinfty$ map from $X$ into $\SB^n$ is homotopic to a regular map.
Therefore the set $\RC(X, \SB^n)$ is dense in the space $\Cinfty(X,
\SB^n)$ by Corollary~\ref{cor-1-5}. The converse is obvious.
\end{proof}

Here is an illuminating example, which itself is a generalization of
Theorem~\ref{th-1-6}.

\begin{example}\label{ex-5-2}
Let $n,k$ be two positive integers and let $\SB^n_{2k}$ be the Fermat
$n$-sphere of degree $2k$,
\begin{equation*}
\SB^n_{2k} = \{(x_0,\ldots,x_n) \in \R^{n+1} : x_0^{2k} + \cdots +
x_n^{2k} = 1\}.
\end{equation*}
Clearly, $\SB^n_{2k}$ is a nonsingular real algebraic variety
diffeomorphic to the unit $n$-sphere ${\SB^n=\SB^n_2}$. If $k$ is odd, then
the set of regular maps $\RC(\SB^n_{2k}, \SB^n)$ is dense in the space
of $\Cinfty$ maps $\Cinfty(\SB^n_{2k}, \SB^n)$. Indeed, the regular map
\begin{equation*}
f \colon \SB^n_{2k} \to \SB^n, \quad (x_0, \ldots, x_n) \mapsto (x_0^k,
\ldots, x_n^k)
\end{equation*}
is a homeomorphism, and hence $\deg(f) = 1$ if $\SB^n_{2k}$ is oriented
in the standard way. Therefore the assertion follows from
Theorem~\ref{th-5-1}.\qed
\end{example}

The following is a variant of Theorem~\ref{th-5-1} for nonorientable
varieties.

\begin{theorem}\label{th-5-3}
Let $X$ be a compact connected nonorientable nonsingular real algebraic
variety of dimension $n$. Then either
\begin{iconditions}
\item\label{th-5-3-i} the set $\RC(X, \SB^n)$ is dense in the space
$\Cinfty(X,\SB^n)$, or

\item\label{th-5-3-ii} the closure of $\RC(X, \SB^n)$ in $\Cinfty(X,
\SB^n)$ coincides with the set of all $\Cinfty$ null homotopic maps from
$X$ into $\SB^n$.
\end{iconditions}
\end{theorem}

\begin{proof}
By Hopf's theorem, there are exactly two homotopic classes of $\Cinfty$
maps from $X$ into $\SB^n$. One of them is represented by a constant map
(which obviously is a regular map). Therefore the proof is complete in
view of Corollary~\ref{cor-1-5}.
\end{proof}

Next we give a relevant example.

\begin{example}\label{ex-5-4}
Let $k$ be a positive integer and let $X_k$ be the blowup of the
$2$-sphere $\SB^2$ at $k$~points. Clearly, $X_k$ is a compact connected
nonsingular real algebraic surface, which is a $\Cinfty$ nonorientable
surface of genus $k$. According to \cite[Theorem~1.7]{bib3}, the set
$\RC(X_k, \SB^2)$ is dense in $\Cinfty(X_k, \SB^2)$ for all $k \geq 1$.
Furthermore, if $k$ is odd, then for every algebraic model~$X$ of $X_k$
the set $\RC(X, \SB^2)$ is dense in $\Cinfty(X, \SB^2)$ by
\cite[Theorem~2]{bib5}. However, if $k$ is even, then there exists an
algebraic model~$Y$ of $X_k$ such that the closure of $\RC(Y,\SB^2)$ in
$\Cinfty(Y, \SB^2)$ consists precisely of all $\Cinfty$ null homotopic
maps from Y into $\SB^2$, see \cite[Theorem~3.3]{bib5}.
\end{example}

It is convenient to bring into play the homotopy groups. Let $n$ be a
positive integer. As a base point in the unit $n$-sphere $\SB^n$ we
choose $s_n=(1,0,\ldots,0)$. For any given real algebraic variety $Y$
with base point $y_0 \in Y$, let $\pialg_n(Y,y_0)$ denote the subset of
the $n$th homotopy group $\pi_n(Y,y_0)$ comprising the homotopy classes
represented by regular maps from $\SB^n$ into~$Y$ that preserve the base
points. We write $\pialg_n(\SB^p)$, $\pi_n(\SB^p)$ instead of
$\pialg_n(\SB^p, s_p)$, $\pi_n(\SB^p,s_p)$, respectively. It is an open
problem whether $\pialg_n(\SB^p)$ is a subgroup of $\pi_n(\SB^p)$ for
all pairs of positive integers $(n,p)$, see \cite[p.~366]{bib2},
\cite{bib56} and Proposition~\ref{prop-5-5} for partial results.

Theorem~\ref{th-1-1} implies the following.

\begin{proposition}\label{prop-5-5}
Let $Y$ be a homogeneous space for a linear real algebraic group $G$ and
let $y_0$ be a point in $Y$. Assume that $\pialg_n(Y,y_0) =
\pi_n(Y,y_0)$ for some positive integer $n$. Then the set of regular
maps $\RC(\SB^n,Y)$ is dense in the space of $\Cinfty$ maps
$\Cinfty(\SB^n, Y)$.
\end{proposition}

\begin{proof}
Let $f \colon \SB^n \to Y$ be a $\Cinfty$ map. Choose an element $a \in
G$ such that $a \cdot f(s_n)=y_0$. Since $\pialg_n(Y,y_0) =
\pi_n(Y,y_0)$, the $\Cinfty$ map $g \colon \SB^n \to Y$, $x \mapsto a
\cdot f(x)$ is homotopic to a regular map, and hence, by
Theorem~\ref{th-1-1}, it can be approximated by regular maps in the
$\Cinfty$ topology. Consequently, the map $f$ can be approximated by
regular maps in the $\Cinfty$ topology because $f(x) = a^{-1} \cdot g(x)$
for all $x \in \SB^n$.
\end{proof}

In particular, for every pair $(n,p)$ of positive integers, the set
$\RC(\SB^n, \SB^p)$ is dense in the space of $\Cinfty$ maps
$\Cinfty(\SB^n, \SB^p)$ if and only if $\pialg_n(\SB^p) = \pi_n(\SB^p)$.
The following result is another generalization of Theorem~\ref{th-1-6}
and provides and additional support for Conjecture~\ref{conj-i}.

\begin{theorem}\label{th-5-6}
Let $(n,p)$ be a pair of positive integers. Then the set of regular maps
$\RC(\SB^n,\SB^p)$ is dense in the space of $\Cinfty$ maps
$\Cinfty(\SB^n,\SB^p)$ in each of the following five cases:
\begin{iconditions}
\item\label{th-5-6-i} $p=1,2$ or $4$.

\item\label{th-5-6-ii} $n-p \leq 3$.

\item\label{th-5-6-iii} $4 \leq n-p \leq 5$ with possible exception for
the pairs $(9,5)$, $(7,3)$, $(11,6)$, $(10,5)$ and $(8,3)$.

\item\label{th-5-6-iv} The homotopy group $\pi_n(\SB^p)$ is finite
cyclic of odd order, and $p$ is odd with $n \leq 2p-2$.

\item\label{th-5-6-v} $n=p+13$, where $p$ is odd and $p \geq 15$.
\end{iconditions}
\end{theorem}

\begin{proof} \ref{th-5-6-i} This is proved in \cite[Theorem~1.1]{bib3}
and has already been mentioned in Section~\ref{sec:1}.

\ref{th-5-6-ii} The case $n=p$ is contained in Theorem~\ref{th-1-6}. For
$1 \leq n-p \leq 3$, one has $\pialg_n(\SB^p) = \pi_n(\SB^p)$ by
\cite[Corollaries 1.2, 1.3  and 1.4]{bib56}. The case $n<p$ is trivial.

\ref{th-5-6-iii} One has $\pi_{p+4}(\SB^p)=0$ for $p \geq 6$ and
$\pi_{p+5}(\SB^p)=0$ for $p \geq 7$ (see \cite[\mbox{pp. 331, 332}]{bib35}), which
together with \ref{th-5-6-i} completes the proof.

\ref{th-5-6-iv} In this case, $\pialg_n(\SB^p) = \pi_n(\SB^p)$ by
\cite[Theorem~2.1 and p.~163]{bib4}.

\ref{th-5-6-v} One has $\pi_{p+13}(\SB^p) = \Z/3$ if $p \geq 15$ (see
\cite[p.~188]{bib59}), so the assertion follows from~\ref{th-5-6-iv}.
\end{proof}

Recall a basic notion from algebraic topology. An \emph{H-space} is a
pointed topological space~$X$ with base point $e$, together with a
continuous map (H-space multiplication) ${\mu \colon X \times X \to X}$
such that $\mu(e,e)=e$, and the maps $X \to X$ defined by
\begin{equation*}
x \mapsto \mu(x,e) \quad \text{and} \quad x \mapsto \mu(e,x)
\end{equation*}
are homotopic to the identity map of $X$ through homotopies that keep
the base point $e$ fixed. For every positive integer $n$, the group
operation in the homotopy group $\pi_n(X) \coloneqq \pi_n(X,e)$ is
induced by the H-space multiplication $\mu$, see \cite[p.~443,
Theorem~4.1]{bib22}.

For our purpose relevant examples of H-spaces are real algebraic groups
(in particular, the unit~sphere $\SB^3$ with the multiplication of the
quaternions of norm~$1$) and the unit sphere~$\SB^7$, the latter with
the multiplication of the octonions of norm~$1$.

\begin{proposition}\label{prop-5-7}
Let $Y$ be either a real algebraic group or the unit sphere $\SB^7$.
Then, for every positive integer~$n$, the subset $\pialg_n(Y)$ is a
subgroup of the homotopy group $\pi_n(Y)$.
\end{proposition}

\begin{proof}
In the case under consideration, the H-space multiplication $Y \times Y
\to Y$, $(a,b) \mapsto ab$ and the inverse operation $Y \to Y$, $a
\mapsto a^{-1}$ are regular maps. The assertion follows since the group
operation in $\pi_n(Y)$ is induced by the H-space multiplication.
\end{proof}

The following example has already been alluded to in connection with
Conjecture~\ref{conj-ii} in Section~\ref{sec:1}.

\begin{example}\label{ex-5-8}
Let $(n,p)$ be a pair of integers, $n>p\geq 1$. Then there exist an
algebraic model $X$ of the $n$-dimensional torus $(\SB^1)^n$ and a
$\Cinfty$ map $f \colon X \to \SB^p$ such that $f$ is not homotopic to
any regular map from $X$ into $\SB^p$.

This assertion can be proved as follows. Let $h \colon (\SB^1)^p \to
\SB^p$ be a $\Cinfty$ map of topological degree~$1$. Then the induced
homomorphism in homology ${h_* \colon H_p((\SB^1)^p; \Z/2) \to H_p(\SB^p;
\Z/2)}$ is an isomorphism. Let $w$ be a generator of the cohomology group
$H^p(\SB^p; \Z/2) \cong \Z/2$; so $w$~corresponds via the Poincar\'e
duality to a point in $\SB^p$. Clearly, if $[(\SB^1)^p] \in
H_p((\SB^1)^p; \Z/2)$ is the fundamental class of $(\SB^1)^p$, then the
Kronecker index $\langle w, h_*([(\SB^1)^p])\rangle$ is nonzero.
Therefore the assertion holds by \cite[Theorem~2.8]{bib16} (with
$K=(\SB^1)^p$, $L=(\SB^1)^{n-p}$, $Y=\SB^p$).

Consequently, the set of regular maps
$\RC(X, \SB^p)$ is not dense in the space of $\Cinfty$ maps $\Cinfty(X,
\SB^p)$.
\end{example}

Next we illustrate the behavior of regular maps from the product of
spheres $\SB^p \times \SB^q$ into the sphere $\SB^{p+q}$.

\begin{example}\label{ex-5-9}
The results depend strongly on the specific values of $p$ and $q$. In
what follows by a degree of a map we mean a topological degree.
\begin{exconditions}[widest=iii]
\item\label{ex-5-9-i} According to \cite[Theorem~12]{bib50}, there
exists a regular map $\SB^4 \times \SB^2 \to \SB^6$ of degree~$1$.
Hence, in view of \cite[Corollary~4.2]{bib4}, for every integer~$d$
there exists a regular map ${\SB^4 \times \SB^2 \to \SB^6}$ of degree~$d$.
Now, it follows from Hopf's theorem that each $\Cinfty$ map $\SB^4
\times \SB^2 \to \SB^6$ is homotopic to a regular map. Consequently, by
Corollary~\ref{cor-1-5}, the set of regular maps $\RC(\SB^4 \times
\SB^2, \SB^6)$ is dense in the space of $\Cinfty$ maps $\Cinfty(\SB^4
\times \SB^2, \SB^6)$.

\item\label{ex-5-9-ii} The same argument shows that $\RC(\SB^4 \times
\SB^1, \SB^5)$ is dense in $\Cinfty(\SB^4 \times \SB^1, \SB^5)$.

\item\label{ex-5-9-iii} By \cite[Theorem~14]{bib50} and
\cite[Corollary~4.2]{bib4}, for every even integer~$d$ there exists a
regular map $\SB^2 \times \SB^2 \to \SB^4$ of degree~$d$. It is an open
problem whether there exists a regular map $\SB^2 \times \SB^2 \to
\SB^4$ of odd degree; if it does, then $\RC(\SB^2 \times \SB^2, \SB^4)$
is dense in $\Cinfty(\SB^2 \times \SB^2, \SB^4)$.

\item\label{ex-5-9-iv} Assuming that both $p$ and $q$ are odd positive
integers, according to Theorem~\ref{th-1-7}, a $\Cinfty$ map $f \colon
\SB^p \times \SB^q \to \SB^{p+q}$ can be approximated by regular maps if
and only if it is null homotopic.
\end{exconditions}
\end{example}

%Section 6
\section{Regular maps into real algebraic groups}\label{sec:6}

For maps into classical groups or Stiefel manifolds (see
Example~\ref{ex-1-4}\ref{ex-1-4-iii} for the notation), we have the
following result supporting Conjecture~\ref{conj-i}.

\begin{proposition}\label{prop-6-1}
Let $(n,m)$ be a pair of positive integers and let $Y$ be one of the
following real algebraic varieties:
\begin{iconditions}
\item\label{prop-6-1-i} $Y=\ON(m)$ or $Y=\SO(m)$ with $n \leq m-2$ and
$n=8k+l$, $k \in \Z$, $l=2,4,5$ or $6$;

\item\label{prop-6-1-ii} $Y = \U(m)$ or $Y=\SU(m)$ with $n \leq 2m-1$
and $n=2k$, $k\in \Z$;

\item\label{prop-6-1-iii} $Y=\Sp(m)$ with $n \leq 4m+1$ and $n=8k+l$, $k
\in \Z$, $l=0,1,2$ or $6$;

\item\label{prop-6-1-iv} $Y=\V_r(\F^m)$ with $n \leq (m-r+1)d(\F)-2$,
where $d(\F) = \dim_{\R}{\F}$.
\end{iconditions}
Then the set of regular maps $\RC(\SB^n,Y)$ is dense in the space of
$\Cinfty$ maps $\Cinfty(\SB^n,Y)$.
\end{proposition}

\begin{proof}
According to Bott's periodicity theorem \cite{bib21} (see also
\cite[Chap.~8, Remark~4.2, and~12.2]{bib36}) the $n$th homotopy group
$\pi_n(Y)$ is trivial in cases \ref{prop-6-1-i}, \ref{prop-6-1-ii} and
\ref{prop-6-1-iii}. Furthermore, by \cite[Chap.~8, Theorem~5.1]{bib36},
the homotopy group $\pi_n(Y)$ is trivial in case \ref{prop-6-1-iv}.
Therefore the proof is complete in view of Corollary~\ref{cor-1-2}.
\end{proof}

Next we consider regular maps from $\SB^n$ into $\U(m)$ or $\SU(m)$, for
$n$ odd with $n \leq 2m-1$. This is harder to handle than
Proposition~\ref{prop-6-1}\ref{prop-6-1-ii} since the homotopy groups
$\pi_n(\U(m))$ and $\pi_n(\SU(m))$ are nontrivial if $n$ is odd ($n \neq
1$ for the latter group). We have the following partial result.

\begin{theorem}\label{th-6-2}
Let $(n,m)$ be a pair of positive integers and let $G(m)$ denote either
$\U(m)$ or $\SU(m)$. Assume that $n=2k-1$ is odd and $m \geq 2^{k-1}$.
Then the set of regular maps $\RC(\SB^n, G(m))$ is dense in the space of
$\Cinfty$ maps $\Cinfty(\SB^n, G(m))$.
\end{theorem}

As a preparation for the proof of Theorem~\ref{th-6-2}, we briefly
summarize the discussion contained in \cite[Part~III B]{bib20}. Let
$k,p$ be two integers with $1 \leq k \leq p$. For any continuous map $f
\colon \SB^{2k-1} \to \GL_p(\CB)$, its degree $\deg(f)$ is an integer
defined as follows. The map~$f$ is homotopic to the composite of some
continuous map $h \colon \SB^{2k-1} \to \GL_k(\CB)$ and the inclusion
map $\GL_k(\CB) \hookrightarrow \GL_p(\CB)$. The first column $h_1$ of
$h$ determines a continuous map $h_1 \colon \SB^{2k-1} \to \CB^k
\setminus \{0\}$, and therefore one gets a continuous map
\begin{equation*}
\psi = \frac{h_1}{\norm{h_1}} \colon \SB^{2k-1} \to \SB^{2k-1},
\end{equation*}
where $\SB^{2k-1}$ is regarded as a subset of $\CB^k=\R^{2k}$. The
topological degree $\deg(\psi)$ of $\psi$ is divisible by $(k-1)!$, and
one sets
\begin{equation*}
\deg{f} = (-1)^{k-1} \frac{\deg(\psi)}{(k-1)!}.
\end{equation*}
The definition of $\deg(f)$ does not depend on the choice of $h$.
Furthermore, one has the following variant of Bott's periodicity
theorem: If $1 \leq k \leq p$, then
\begin{equation*}
\deg \colon \pi_{2k-1}(\GL_p(\CB)) \to \Z, \quad [f] \mapsto \deg(f)
\end{equation*}
is a well-defined group isomorphism.

As usual, for any complex matrix~$A$, let $A^*$ denote its conjugate
transpose. For any positive integer $r$, denote by $I_r$ the identity
$r$-by-$r$ matrix.

Suppose given two continuous maps $f \colon \SB^{2k-1} \to \GL_p(\CB)$
and $g \colon \SB^{2l-1} \to \GL_q(\CB)$, where $1 \leq k \leq p$ and $1
\leq l \leq q$. The product
\begin{equation*}
f \# g \colon \SB^{2(k+l)-1} \to \GL_{2pq}(\CB)
\end{equation*}
is a continuous map defined by
\begin{equation*}
(x,y) \mapsto 
\begin{pmatrix}
F(x) \otimes I_q & -I_p \otimes G(y)^*\\
I_p\otimes G(y) & F(x)^*\otimes I_q
\end{pmatrix},
\end{equation*}
where $x \in \CB^k$, $y \in \CB^l$, $\norm{(x,y)} = 1$, while $F$ and
$G$ are the homogeneous extensions of $f$ and $g$, respectively. To be
precise,
\begin{equation*}
F \colon \CB^k \to \Mat_p(\CB), \quad F(x) =
\begin{cases}
\norm{x}f\big(\frac{x}{\norm{x}}\big) &\text{for } x \in \CB^k \setminus \{0\},\\
0 &\text{for } x=0,
\end{cases}
\end{equation*}
where $\Mat_p(\CB)$ is the space of all complex $p$-by-$p$ matrices, and
$G \colon \CB^l \to \Mat_q(\CB)$ is defined analogously. One has the
formula
\begin{equation*}
\deg(f \# g) = \deg(f) \deg(g).
\end{equation*}

Now, starting with the map
\begin{equation*}
a \colon \SB^1 \to \GL_1(\CB) = \CB \setminus \{0\}, \quad z \mapsto z,
\end{equation*}
we define a sequence of continuous maps
\begin{equation*}
a_k \colon \SB^{2k-1} \to \GL_{2^{k-1}}(\CB), \quad k=1,2,\ldots
\end{equation*}
by a recursive formula: $a_1=a$ and $a_k = a_{k-1} \# a$ for $k \geq 2$.
Since $\deg(a) =1$, it follows that $\deg(a_k) =1$ for all $k \geq 1$.
Thus, by the variant of Bott's periodicity theorem stated above, for
every $k \geq 1$ the homotopy group $\pi_{2k-1}(\GL_{2^{k-1}}(\CB)) \cong
\Z$ is generated by the homotopy class represented by $a_k$.

The maps $a_k$ have some other useful properties. Let $A_k \colon \CB^k
\to \Mat_{2^{k-1}}(\CB)$ be the homogeneous extension of $a_k$, $k \geq
1$. Clearly, $A_1(z_1)=z_1$ and
\begin{equation*}
A_k(z_1,\ldots,z_k) =
\begin{pmatrix}
A_{k-1}(z_1,\ldots, z_{k-1}) & -\bar z_k I_{2^{k-2}} \\
z_k I_{2^{k-2}} & A_{k-1}(z_1,\ldots,z_{k-1})^*
\end{pmatrix}
\quad \text{for } k \geq 2,
\end{equation*}
where $(z_1, \ldots, z_k) \in \CB^k$ and $\bar z_k$ is the conjugate of
$z_k$. It follows that $A_k$ is an $\R$-linear map for $k \geq 1$.
Furthermore,
\begin{gather*}
A_k(z_1,\ldots,z_k) A_k(z_1,\ldots,z_k)^* = \Big(\sum_{j=1}^k z_j
\bar z_j \Big) I_{2^{k-1}} \quad \text{for } k \geq 1, \\
\det A_1(z_1) = z_1 \bar z_1, \quad
\det A_k(z_1,\ldots,z_k) = \Big(\sum_{j=1}^k z_j \bar z_j
\Big)^{2^{k-2}} \quad \text{for } k \geq 2.
\end{gather*}
In particular, $a_k$ can be regarded as a regular map $a_k \colon
\SB^{2k-1} \to \SU(2^{k-1})$ for $k \geq 2$.

\begin{proof}[Proof of Theorem~\ref{th-6-2}]
By Proposition~\ref{prop-5-5}, it is sufficient to show that
$\pialg_{2k-1}(G(m)) = \pi_{2k-1}(G(m))$.

\begin{case}\label{c1}
Suppose that $(k,m) = (1,m)$ with $m\geq 1$. It is well known that
${\pi_1(\SU(m)) = 0}$, and the inclusion map $\U(1) \hookrightarrow \U(m)$
induces an isomorphism $\pi_1(\U(1)) \cong \pi_1(U(m))$. Case~\ref{c1}
follows since $\U(1) = \SB^1$ and $\pialg_1(\SB^1)=\pi_1(\SB^1)$.
\end{case}

\begin{case}\label{c2}
Suppose that $k \geq 2$ and $m \geq 2^{k-1}$. The inclusion map
$G(2^{k-1}) \hookrightarrow \GL_{2^{k-1}}(\CB)$ induces an isomorphism
of the homotopy groups $\pi_{2k-1}(G(2^{k-1}))$ and
$\pi_{2k-1}(\GL_{2^{k-1}}(\CB))$. As explained above, the latter group is
generated by the homotopy class represented by the map $a_k$. Since
$a_k$ is a regular map with values in $G(2^{k-1})$, we get
$\pialg_{2k-1}(G(2^{k-1})) = \pi_{2k-1}(G(2^{k-1}))$ by
Proposition~\ref{prop-5-7}. Moreover, the inclusion map $G(2^{k-1})
\hookrightarrow G(m)$ induces an isomorphism of the homotopy groups
$\pi_{2k-1}(G(2^{k-1}))$ and $\pi_{2k-1}(G(m))$ (see \cite[Chap.~8, Remark~4.2]{bib36}), and hence $\pialg_{2k-1}(G(m)) =
\pi_{2k-1}(G(m))$.\qedhere
\end{case}
\end{proof}

We have one more result on regular maps into the groups $\U(m)$ or
$\SU(m)$.

\begin{proposition}\label{prop-6-3}
Let $(n,m)$ be a pair of positive integers and let $G(m)$ denote either
$\U(m)$ or $\SU(m)$. Assume that $1 \leq n \leq 6$. Then the set of
regular maps $\RC(\SB^n, G(m))$ is dense in the space of $\Cinfty$ maps
$\Cinfty(\SB^n, G(m))$, possibly with the exception of $(n,m)=(5,3)$ or
$(n,m)=(6,3)$.
\end{proposition}

\begin{proof}
Suppose that $(n,m)\neq(5,3)$ and $(n,m)\neq(6,3)$. Then the only cases
not already covered by Proposition~\ref{prop-6-1} and
Theorem~\ref{th-6-2} are $(n,m)$ equal to $(4,2)$, $(5,2)$ and $(6,2)$.
Hence, by Proposition~\ref{prop-5-5}, it is sufficient to show that
$\pialg_n(G(2)) = \pi_n(G(2))$ for $n=4,5$ and~$6$. This is proved in
\cite[Corollaries~1.2, 1.3 and 1.4]{bib56} for $G(2)=\SU(2)=\SB^3$. The
case $G(2)=\U(2)$ follows since the inclusion map $\SU(m)
\hookrightarrow \U(m)$ induces for all $i \geq 2$ an isomorphism between
the corresponding $i$th homotopy groups, see \cite[Chap.~8,
12.2]{bib36}.
\end{proof}

\begin{acknowledgements}
The second named author was partially supported by the National Science Center (Poland) under grant number 2018/31/B/ST1/01059.

We are grateful to Olivier Benoist, J\'anos Koll\'ar and Olivier Wittenberg for very
useful comments.
\end{acknowledgements}

%Appendix
\appendix
\section*{Appendix\\\large\normalfont J\'anos Koll\'ar}
\addcontentsline{toc}{section}{Appendix}
\setcounter{section}{1}
\setcounter{theorem}{0}

The following  unirationality result was conjectured by 
Bochnak and Kucharz.

\begin{proposition}\label{appendix.1} Let $G$ be  a linear real algebraic group of dimension $n$. Then there is  a real algebraic morphism  $g\colon\R^n\to G$ such that  $g(0)=e$ (the group identity)  and the tangent map
\begin{equation}\label{eq-a1}
dg_0\colon \R^n\cong T_0\R^n\to T_eG\cong \R^n\quad\text{is an isomorphism.} 
\end{equation}
\end{proposition}

The original version of the paper used a theorem of Chevalley to prove a variant, where $g$ was defined on a dense, open subset of $\R^n$. The above stronger form does not change any of the main theorems of the paper, but allows to formulate some intermediate steps in a more elegant way and illustrates more precisely unirationality of linear real algebraic groups. It is worth mentioning that the question whether all such groups are rational remains open.

First we show how to obtain Proposition~\ref{appendix.1} from Chevalley's theorem.  Next, looking at Chevalley's method gives a more direct construction.
Then we recall the Cayley transform  which gives a birational morphisms
for the orthogonal and unitary groups.

\begin{say}[Proof using Chevalley's theorem]\label{Kollar-2} By \cite{bib24} 
there is a rational map
$\tilde g\colon \R^n\map G$ satisfying \eqref{eq-a1}.
Thus it is enough to write down a  morphism $h\colon \R^n\to \R^n$
whose image lands in the open set where $\tilde g$ is defined, and set
$g \coloneqq \tilde g \circ h$. We can take 
\begin{equation*}
h\colon {\mathbf x}\mapsto c\cdot \tfrac{{\mathbf x}}{1+||{\mathbf x}||^2},
\end{equation*} 
which sends $\R^n$ to a small neighborhood of  the origin for $0<c\ll 1$. \qed
\end{say}

\begin{say}[Proof using Chevalley's method]
 We use the Iwasawa decomposition to write
$G=K\times (\R^\times)^r\times U$
where $K$ is compact and $U$ is unipotent. 

The Lie algebra $L_U$ of $U$ is nilpotent, so the exponential map
$\exp\colon L_U\to U$ is algebraic and even a homeomorphism. 
The $\R^\times$ factor is the least natural, but we can use
$g\colon \R\to \R^\times$ given by $t\mapsto (1+t+t^2)(1+t^2)^{-1}$.

It remains to deal with $K$. For any root $\alpha$ we have a subgroup
 $\SB^1\cong C_\alpha\subset K$.  Composing it with the inverse of the
 stereographic projection from $-1$   we get
$g_\alpha\colon \R^1\to \SB^1\cong C_\alpha\subset K$.

Now fix an ordering of the roots and send
$v=\sum c_\alpha \alpha$ to
\begin{equation*}
    \pushQED{\qed}
    g(v)\coloneqq\tprod_{\alpha} g_\alpha(c_\alpha)\in K. \qedhere
    \popQED
\end{equation*}
Note that $g$ does depend on the ordering of the roots. It would be interesting to find a more natural construction that generalizes the Cayley transform.  
\end{say}

%Next we discuss some cases where $g$ can be chosen birational.

\begin{example}[Cayley transform] Let $A$ be a skew-symmetric $n\times n$ matrix.  Then $I + A$ is invertible, and the Cayley transform is
$ Q(A)\coloneqq(I-A)(I+A)^{-1}$.
Then $Q(A)$ is  an orthogonal matrix, giving the required map
$Q\colon T_0\SO_n\to \SO_n$.

Similarly, starting with a skew-Hermitian matrix $A$  we get
$Q\colon T_0\U_n\to \U_n$.
\end{example}

% \begin{exmp}[Spheres]\label{appendix.2} Let $\SB^n\subset \R^{n+1}$ be the unit sphere
% with tangent bundle $T\SB^n$. For $p\in s^n$ and $v\in T_p\SB^n$ let $g(v)$  be the intersection point of the line connecting $p+v$ with $-p$ with $\SB^n$.
% That is, the restriction of $g$ to   $T_p\SB^n$ is  the inverse of the
% stereographic projection from $-p$. 
% \end{exmp}

\bigskip
\noindent J\'anos Koll\'ar\\
Princeton University\\
Princeton NJ 08544-1000 USA\\
\emailaddrname: \texttt{kollar@math.princeton.edu}

%References

%\cleardoublepage
\phantomsection

\end{document}